\documentclass[reqno]{amsart}

\usepackage{amsfonts}
\usepackage{amssymb}
\usepackage{amsmath}
\usepackage{mathrsfs}
\usepackage{amsthm}
\usepackage[colorlinks,linkcolor=blue,citecolor=blue]{hyperref}

 \newtheorem{theorem}{Theorem}[section]

\newtheorem{lemma}[theorem]{Lemma}
\newtheorem{proposition}[theorem]{Proposition}

\theoremstyle{definition}

\numberwithin{equation}{section}

\begin{document}

\title[An exterior overdetermined problem in convex cones]{An exterior overdetermined problem for Finsler $N$-laplacian in convex cones}

\author[G. Ciraolo]{Giulio Ciraolo}
\address[G. Ciraolo]{Department of Mathematics ``Federigo Enriques"\\
Universit\`a degli Studi di Milano\\ Via Cesare Saldini 50, 20133 Milan\\
Italy}
\email{giulio.ciraolo@unimi.it}

\author[X. Li]{Xiaoliang Li}
\address[X. Li]{School of Mathematical Sciences\\
Beijing Normal University\\
	100875 Beijing\\
	P.R. China}
\email{rubiklixiaoliang@163.com}

\thanks{}
\subjclass[2010]{35N25, 35A23, 35B06, 31B15}

\keywords{Overdetermined problems; mixed boundary conditions; Wulff shapes; convex cones; anisotropic isoperimetric inequality}

\begin{abstract}
We consider a partially overdetermined problem for anisotropic $N$-Laplace equations in a convex cone $\Sigma$ intersected with the exterior of a bounded domain $\Omega$ in $\mathbb{R}^N$, $N\geq 2$. Under a prescribed logarithmic condition at infinity, we prove a rigidity result by showing that the existence of a solution implies that $\Sigma\cap\Omega$ must be the intersection of the Wulff shape and $\Sigma$. Our approach is based on a Pohozaev-type identity and the characterization of minimizers of the anisotropic isoperimetric inequality inside convex cones.
\end{abstract}

\maketitle

\section{Introduction}

In this paper we consider a variational problem in an anisotropic medium which is related to the so-called conformal $N$-capacity (or logarithmic capacity). Our main goal is to provide symmetry results for a partially overdetermined problem in convex cones. 

The physical motivation for studying anisotropic problems comes from well-established models of surface energy (see for instance \cite{Taylor1978}). Moreover, there are many mathematical interesting aspects arising when one considers symmetry problems in an anisotropic setting (\cite{BCS2016,Bianchini-Ciraolo-2018,CRS2016,Cianchi-Salani-2009,CFR2020,Cozzi2014,Cozzi2016,DPV2021}).

The logarithmic capacitance has applications in physics, such as in the classical study of nonlinear voltage and  capacitance differences between, respectively, diodes and capacitors. It appears in condensed-matter and high-energy physics \cite{Goldenfeld,Ramond}, and it is naturally studied in quasiconformal geometry (see for instance \cite{Gehring}). The study of symmetry problems in convex cones has recently attracted the interest of many authors, see for instance \cite{CRS2016,CFR2020,Ciraolo2020,DPV2021,Figalli-Indrei-2013,Lions-Pacella-1990,Pacella-Tralli-2020,RR2004}. As far as the authors know, overdetermined capacity problems in convex cones have not been considered so far, even in the Euclidean case.

\subsection{The mathematical framework} Let $H\in C^2(\mathbb{R}^N\setminus\{0\})$ be a positively homogeneous function of degree one which is positive on $\mathbb{S}^{N-1}$, and let $H_0$ be the dual function of $H$ defined by
\begin{equation}\label{intro-eq:def-H0}
H_0(x):=\sup_{H(\xi)=1}\left<x,\xi\right>,\quad\text{for }x\in\mathbb{R}^N.
\end{equation} 

Given an open, convex cone $\Sigma$ and a bounded domain $\Omega$ in $\mathbb{R}^N$, we consider the following exterior boundary value problem for the anisotropic $N$-Laplace equation
\begin{equation}\label{eq:log-cone}
\left\{
\begin{array}{ll}
\Delta_N^H u=0 & \text{in }\Sigma\cap\overline{\Omega}^c,\\
u=0 & \text{on }\Gamma_0,\\
\left<a(\nabla u),\nu\right>=0 & \text{on }\Gamma_1,
\end{array}
\right.
\end{equation}
with prescribed asymptotic behavior at infinity
\begin{equation}\label{eq:asym-log-cone}
\frac1d\leq\frac{u(x)}{\ln H_0(x)}\leq d\ \text{ for some }d>1,\quad\text{as }H_0(x)\to\infty.
\end{equation}
Here $\Delta_N^H$ is the so-called Finsler $N$-Laplacian which is given by 
$$
\Delta_N^H u=\mathrm{div}\left(a(\nabla u)\right)
$$ 
in the sense of distributions, where
$$
a(\xi) = \frac{1}{N}\nabla H^N(\xi) \quad \forall \xi \in \mathbb{R}^N \,,
$$
\begin{equation*}
\overline{\Omega}^c:=\mathbb{R}^N\setminus\overline{\Omega},\quad\Gamma_0:=\Sigma\cap\partial\Omega,\quad\Gamma_1:=\overline{\Omega}^c\cap\partial\Sigma,\\ 
\end{equation*}
and $\nu$ is the outer normal to $\partial\Sigma$. We will assume throughout this paper that $\Sigma\cap\overline{\Omega}^c$ is connected and $\mathcal{H}^{N-1}(\Gamma_0)>0$ where $\mathcal{H}^{N-1}$ stands for the $(N-1)$-dimensional Hausdorff measure. By a weak solution of \eqref{eq:log-cone} we mean a function $u\in W_{loc}^{1,N}(\overline{\Sigma}\cap\Omega^c)$ with $u=0$ on $\Gamma_0$,\footnote{Given a bounded open set $O\subset\mathbb{R}^N$, a function $f\in W^{1,N}(O)$ and a relatively open subset $\Gamma$ of $\partial O$, we say that $f=0$ on $\Gamma$ if $f$ is the limit in $W^{1,N}(O)$ of a sequence of functions in $C^\infty(\overline{O})$ vanishing in a neighborhood of $\overline{\Gamma}$. In particular, when $\partial O$ is Lipschitz, this definition is equivalent to $T_{\Gamma}f=0$, where $T_\Gamma f$ is the trace of $f$ on $\Gamma$ (see for instance \cite[Proposition 7.86]{Salsa2016}).} such that
$$\int_{\Sigma\cap\overline{\Omega}^c}\left<a(\nabla u),\nabla\varphi\right>dx=0$$
for all $\varphi\in W^{1,N}(\Sigma\cap\overline{\Omega}^c)$ with $\varphi=0$ on $\Gamma_0$ and with bounded support.

Notice that, when $\Sigma=\mathbb{R}^N$, we have $\Gamma_1=\emptyset$ and the third condition in \eqref{eq:log-cone} is trivially satisfied. In this case, if $H$ is the Euclidean norm (i.e. $H(\xi)=|\xi|$), the model \eqref{eq:log-cone}--\eqref{eq:asym-log-cone} applies to the study of logarithmic capacity \cite{Andrea-Paola-2005,Xiao2020}, and it determines the $N$-equilibrium potential of $\Omega$, which naturally appears in computing the capacitance difference between coaxial cylindrical capacitors (see \cite{Polya-Szego-1951}). Analogously, for a general $H$, problem \eqref{eq:log-cone}--\eqref{eq:asym-log-cone} can be applied to the study of the related capacity problems, when the set $\Omega$ is embedded in a possibly anisotropic medium. In this connection, problem \eqref{eq:log-cone}--\eqref{eq:asym-log-cone} can also be seen as a logarithmic counterpart of those arising from the study  of anisotropic $p$-capacity with $1<p<N$ instead, see for instance \cite{BCS2016,Bianchini-Ciraolo-2018,XY2021}. More concrete applications of related anisotropic models arise in the theory of crystals as well as in noise-removal procedures in digital image processing, see for instance \cite{EO2004,Taylor1978} and the references therein.

\subsection{The overdetermined problem} The aim of this paper is to characterize the shape of $\Omega$ in terms of the existence of solutions to problem \eqref{eq:log-cone}--\eqref{eq:asym-log-cone}, coupled with the overdetermined condition
\begin{equation}\label{eq:over-Neumann}
H(\nabla u)=C\quad\text{on }\Gamma_0,
\end{equation}
for some given constant $C>0$. Differently from the classical overdetermined problems originated in Serrin's famous work \cite{Serrin1971}, we emphasize that whenever $\Sigma\subsetneq\mathbb{R}^N$, problem \eqref{eq:log-cone}--\eqref{eq:asym-log-cone} with \eqref{eq:over-Neumann} is partially overdetermined, since in this case both Dirichlet and Neumann conditions are simultaneously imposed only on a part of the boundary of $\Omega$, namely $\Gamma_0$. Accordingly, one may expect to determine the shape of $\Gamma_0$ while a sole homogeneous Neumann boundary condition is assigned on $\Gamma_1$, provided that $\Gamma_1$ satisfies some geometric constrain.

It is well-known that in the isotropic setting, the characterized geometric property of the domain in Serrin's overdetermined problems is the spherical symmetry. In an anisotropic framework encoded by the anisotropy $H$, the natural counterpart of this feature is the so-called Wulff shape which is a ball in the dual function $H_0$, that is
\begin{equation}\label{intro-eq:def-BH0}
B_R^{H_0}(x_0):=\{x\in\mathbb{R}^N:H_0(x-x_0)<R\},
\end{equation}
where $x_0\in\mathbb{R}^N$ and $R>0$ denote the center and the radius, respectively. Starting from this connection, our research is motivated by the observation that there is indeed an explicit logarithmic function in $H_0$ fulfilling problem \eqref{eq:log-cone}--\eqref{eq:asym-log-cone} with \eqref{eq:over-Neumann} in the classical sense, when $\Omega=B_R^{H_0}(x_0)$ for suitable choice of $x_0$ depending on the convex cone $\Sigma$. To make this precise, we need some more notation.

In general, up to a change of coordinates, we can write $\Sigma=\mathbb{R}^k\times\tilde{\Sigma}$, where $k\in\{0,\cdots,N\}$ and $\tilde{\Sigma}\subset\mathbb{R}^{N-k}$ is an open, convex cone with vertex at the origin which does not contain a line. Besides, following \cite{Cozzi2014,Cozzi2016}, we say that the function $H$ is uniformly elliptic if its $1$-sublevel set 
\begin{equation}\label{intro-eq:def-B1H}
B_1^{H}:=\{\xi\in\mathbb{R}^N:H(\xi)<1\}\end{equation}
is uniformly convex. This is a standard assumption ensuring the Hessian of $H^N$ is positive definite in $\mathbb{R}^N\setminus\{0\}$, and thus equation \eqref{eq:log-cone} is of elliptic type, though possibly degenerate. Moreover, under this assumption, one has that $H$ is convex and $H_0\in C^2(\mathbb{R}^N\setminus\{0\})$. These claims will be explained in detail in Subsection \ref{subsec:pro-H}.

With above notations, observe that
\begin{proposition}\label{pro:observe}
Let $N\geq 2$ and $\Sigma=\mathbb{R}^k\times\tilde{\Sigma}$ where $k\in\{0,\cdots,N\}$ and $\tilde{\Sigma}\subset\mathbb{R}^{N-k}$ is an open, convex cone with vertex at the origin which contains no lines. 

Let $R>0$, $x_0\in\mathbb{R}^k\times\{0\}$ and $\Omega=B_R^{H_0}(x_0)$. Assume that $H\in C^2(\mathbb{R}^N\setminus\{0\})$ is uniformly elliptic. Then the function
\begin{equation}\label{intro-eq:formula-u}
u(x)=CR\ln\frac{H_0(x-x_0)}{R}
\end{equation} 
is of class $C^2(\mathbb{R}^N\setminus\{x_0\})$ and is the unique solution to problem \eqref{eq:log-cone}--\eqref{eq:asym-log-cone} with \eqref{eq:over-Neumann}.
\end{proposition}
Proposition \ref{pro:observe} can be verified directly by exploiting the homogeneity and differentiability of $H$ and $H_0$, see Lemmas \ref{pre-lem:pro-H} and \ref{pre-lem:pro-H-con} below. In particular, the uniqueness of \eqref{intro-eq:formula-u} follows from the well-posedness of original problem \eqref{eq:log-cone}--\eqref{eq:asym-log-cone}, as we will show in Theorem \ref{cap-thm:F-log}. 

\subsection{The main result} Concerning the reverse of Proposition \ref{pro:observe}, an investigation into seeking a Wulff shape characterization for the domain $\Omega$ naturally arises. Indeed, we shall prove that if problem 
\eqref{eq:log-cone}--\eqref{eq:asym-log-cone} with \eqref{eq:over-Neumann} admits a weak solution, then $\Sigma\cap\Omega$ must be $\Sigma\cap B_R^{H_0}(x_0)$, under some regularity assumption on $\Omega$ and on the solution. More precisely, our main result is the following.


\begin{theorem}\label{thm:over-log-cone}
Let $\Omega$ be a bounded domain of $\mathbb{R}^N$ with boundary of class $C^{1,\alpha}$, $\alpha\in(0,1)$ and $N\geq 2$. Let $\Sigma=\mathbb{R}^k\times\tilde{\Sigma}$ where $k\in\{0,\cdots,N\}$ and $\tilde{\Sigma}\subset\mathbb{R}^{N-k}$ is an open, convex cone with vertex at the origin and containing no lines. 

Let $H\in C^2(\mathbb{R}^N\setminus\{0\})$ be a positively homogeneous function of degree one which is positive on $\mathbb{S}^{N-1}$ and uniformly elliptic, and let $H_0$ be its dual function defined by \eqref{intro-eq:def-H0}. 

Assume that there exists a solution $u\in W_{loc}^{1,N}(\overline{\Sigma}\cap\Omega^c)$ to problem \eqref{eq:log-cone}--\eqref{eq:asym-log-cone} with \eqref{eq:over-Neumann} such that 
$$
\nabla u\in L_{loc}^{\infty}(\overline{\Sigma}\cap\Omega^c) \,.
$$ 

Then $\Sigma\cap\Omega=\Sigma\cap B_R^{H_0}(x_0)$ and $u$ is given by \eqref{intro-eq:formula-u}, for some $R>0$ and $x_0\in\mathbb{R}^k\times\{0\}$.
\end{theorem}

In order to compare Theorem \ref{thm:over-log-cone} to related results available in literature, we make some remark concerning the regularity of the solution and on the type of anisotropy we are considering. 

We first remark that, under the above assumptions on $\Omega$ and $H$, $u$ belongs to $C^1((\Sigma\cap\overline{\Omega}^c)\cup\Gamma_0)$. Actually, by interior regularity results in \cite{DiBen1983,Serrin1964,Tolksdorf1984}  a weak solution to equation \eqref{eq:log-cone} which is in $W_{loc}^{1,N}(\Sigma\cap\overline{\Omega}^c)$ is automatically of class $C^1$. Thanks to \cite{Lieberman1988},\footnote{In order to apply the boundary regularity results in \cite{Lieberman1988} (see Theorems 1 and 2 there), it is not difficult to observe that $u\in L_{loc}^\infty((\Sigma\cap\overline{\Omega}^c)\cup\Gamma_0\cup\Gamma_1)$ by employing the Moser iteration argument as in \cite{Serrin1964}.} such regularity can be pushed up to $\Gamma_0$ and to the $C^{1,\alpha}$-regular portion of $\Gamma_1$. Moreover, as we will see in Proposition \ref{pre-pro:second-order} below, 
$$a(\nabla u)\in W_{loc}^{1,2}((\Sigma\cap\overline{\Omega}^c)\cup\Gamma_1) \,.$$
However, the regularity up to the whole boundary is a delicate issue, since it strongly relies on how $\Omega$ and $\Sigma$ intersect. Nevertheless, our result holds true without requiring any global regularity other than 
$$
\nabla u\in L_{loc}^{\infty}(\overline{\Sigma}\cap\Omega^c) \,.
$$ 
This can be viewed as a glueing condition between the cone and $\Gamma_0$ as explained in \cite{Pacella-Tralli-2020}. We also refer this to \cite{CM2011}, where global Lipschitz regularity is proved for Dirichlet or Neumann boundary value problems of $p$-Laplace type in convex domains. We notice that in the case $\Sigma=\mathbb{R}^N$ we do not need to impose additional regularity assumptions on the solution $u$. Moreover, regarding the anisotropy $H$, we note that here we do not assume $H$ to be even, so, in general, $H(\xi)\neq H(-\xi)$; namely, $H$ is not necessarily a norm. The same is true for the dual function $H_0$ as well.

When $k=N$ (i.e. $\Sigma=\mathbb{R}^N$) and $H$ is the Euclidean norm, Theorem \ref{thm:over-log-cone} was proved in Reichel \cite{Reichel1996} and Poggesi \cite{Poggesi2019}, respectively by the well-known moving planes method and by using some integral identities combined with the classical isoperimetric inequality. In an anisotropic setting, the moving plane technique is no more helpful, and the Wulff shape characterization result in the case $k=N$ of Theorem \ref{thm:over-log-cone} was recently treated in Xia--Yin \cite{Xia-Yin-2021} by adapting the arguments used in \cite{Poggesi2019}, under the assumptions that $H\in C^\infty(\mathbb{R}^N\setminus\{0\})$ is a norm and that $\partial\Omega$ is of class $C^{2,\alpha}$. In the present paper, by exploiting similar integral methods, we generalize such characterization results to the setting of convex cones, and also classify the resulting symmetry for the solutions.

Here it is worth pointing out that the generalization is not trivial. Indeed, in the case $\Sigma\subsetneq\mathbb{R}^N$ we are concerned with the mixed boundary value problem \eqref{eq:log-cone}. In order to deal with a problem of this type, we need to establish qualitative properties such as comparison principles and Liouville-type results in a cone setting with homogeneous Neumann boundary conditions. These are essential tools in analyzing the precise asymptotic behavior of the solutions at infinity, since neither Kelvin type transform nor Hopf's boundary point lemma are available in our case. Moreover, due to the lack of global $C^1$-regularity of the solution as well as the non-smoothness of $\Sigma$, we have to employ careful approximation arguments to validate certain integral identities and inequalities, including a Pohozaev-type identity contained in Theorem \ref{P-thm:Pohozaev} and local $W^{1,2}$ estimates up to $\partial\Sigma$ for the nonlinear vector fields of the gradient of the solutions (see Proposition \ref{pre-pro:second-order}). In particular, this   brings a subtle issue that we need to approximate the possibly non-Lipschitz sets in $\Sigma\cap\overline{\Omega}^c$ by Lipschitz sets converging in some boundary integrals.

The result of Theorem \ref{thm:over-log-cone} is strictly related to the anisotropic isoperimetric inequality inside convex cones, which was obtained in Cabr\'e--Ros-Oton--Serra \cite{CRS2016} along with a general weighted version. It states that, in our notation, if $E\subset\mathbb{R}^N$ is a measurable set with finite Lebesgue measure $\mathcal{H}^N$ in $\Sigma$, then
\begin{equation}\label{intro-eq:iso-ineq}
\frac{P_{H}(E;\Sigma)}{\mathcal{H}^N(\Sigma\cap E)^{\frac{N-1}{N}}}\geq \frac{P_H(B_1^{H_0};\Sigma)}{\mathcal{H}^N(\Sigma\cap B_1^{H_0})^{\frac{N-1}{N}}}
\end{equation}
and the equality holds whenever $\Sigma\cap E=\Sigma\cap B_R^{H_0}(x_0)$; here $B_1^{H_0}:=B_1^{H_0}(0)$ is the unit Wulff ball centered at the origin and $P_{H}(E;\Sigma)$ denotes the anisotropic perimeter of $E$ relative to $\Sigma$ defined in \eqref{pre-eq:def-a-perimeter} below. Inequality \eqref{intro-eq:iso-ineq} was proved in \cite{CRS2016} by reducing it to a degenerate case of the classical Wulff inequality which is well-known in the literature (see for instance \cite{Dacorogna1992,Fonseca1991,Taylor1978,FMP2010}). Such an idea is first observed by Figalli and Indrei \cite{Figalli-Indrei-2013} in order to establish a quantitative version of the isoperimetric inequality in convex cones due to Lions--Pacella \cite{Lions-Pacella-1990}, corresponding to \eqref{intro-eq:iso-ineq} where $H$ is the Euclidean norm. Recently, the ideas in \cite{Figalli-Indrei-2013} were further adapted in Dipierro--Poggesi--Valdinoci \cite{DPV2021} to prove the uniqueness of the minimizers of \eqref{intro-eq:iso-ineq} for a general norm $H$, illustrating that the equality in \eqref{intro-eq:iso-ineq} holds if and only if $\Sigma\cap E=\Sigma\cap B_R^{H_0}(x_0)$. However, the same argument still works in our case where $H$ is a gauge as required in \cite{CRS2016}. We shall state this precisely in Theorem \ref{pre-thm:Wulff-cone}, from which we are able to conclude the proof of Theorem \ref{thm:over-log-cone} (see below for a detailed description).

Before explaining the main ideas of our proof, we would like to mention more related studies on Serrin's overdetermined problems for anisotropic equations and their variants in cones. 

In the whole $\mathbb{R}^N$, classical Serrin's result in \cite{Serrin1971} has been extended to the setting of Finsler $p$-Laplacian ($p>1$) both in bounded domains and exterior domains, we refer to \cite{BCS2016,Bianchini-Ciraolo-2018,Cianchi-Salani-2009,Wang-Xia-2011,Xia-Yin-2021} and the references therein. The two alternative approaches used in these literatures are both based on integral identities and they are inspired by the idea of Weinberger \cite{Wein1971} and that of Brandolini--Nitsch--Salani--Trombetti \cite{BNST2008}, respectively. The main difference is that the latter relies on a Cauchy--Schwarz inequality about the Hessian matrix and does not invoke a maximum principle for a $P$-function as introduced in \cite{Wein1971}.

Regarding the variants for cones, rigidity results of Serrin type were first obtained in Pacella--Tralli \cite{Pacella-Tralli-2020}, where they considered an interior overdetermined problem inside a smooth convex cone and gave a characterization of spherical sectors following the approaches in \cite{BNST2008,Wein1971}. Then the first author and Roncoroni \cite{Ciraolo2020} generalized that to more general elliptic operators which are possibly degenerate as well as to space forms. More generally, during the last decade, much interest has been devoted to other parallel problems in convex cones and anisotropic setting (see for instance \cite{CRS2016,CFR2020,Cozzi2014,DPV2021,RR2004}). However, as far as we know, exterior overdetermined problems in unbounded domains contained in cones have not been studied yet even in the isotropic setting. The study presented in this paper may serve as a starting point in this direction.

Now we comment the proof of Theorem \ref{thm:over-log-cone}. Unlike those developed in \cite{Bianchini-Ciraolo-2018,BNST2008,Wein1971}, here we adopt an isoperimetric argument to prove Theorem \ref{thm:over-log-cone} in the spirit of \cite{DPV2021,Esposito2018,Poggesi2019}. The crucial point consists in using integral identities to show that $\Sigma\cap\Omega$ is a minimizer of the anisotropic isoperimetric inequality inside $\Sigma$ (see \eqref{intro-eq:iso-ineq}) whenever problem \eqref{eq:log-cone}--\eqref{eq:asym-log-cone} with \eqref{eq:over-Neumann} admits a solution, which implies the desired Wulff shape characterization. To this aim, the proof is made in three steps. Firstly, via scaling arguments we improve the logarithmic behavior of $u$ prescribed in \eqref{eq:asym-log-cone} and obtain its asymptotic expansion at infinity, see Proposition \ref{cap-pro:asym-equi}. Secondly, we derive a Pohozaev-type identity for equation \eqref{eq:log-cone}. Finally, by an approximation argument we compute the explicit value of the constant $C$ appearing in overdetermined condition \eqref{eq:over-Neumann} and we further apply the Pohozaev identity to deduce that $\Sigma\cap\Omega$ satisfies the equality case of \eqref{intro-eq:iso-ineq}.

The paper is organized as follows. In Section \ref{sec:prelimi}, we collect some  auxiliary and technical results, including properties of the anisotropy $H$, a weak comparison principle and the anisotropic isoperimetric inequality in convex cones. In particular, a second-order regularity result for weak solutions to problem \eqref{eq:log-cone} is established in Subsection \ref{subsec:second-order}. The solvability of problem \eqref{eq:log-cone}--\eqref{eq:asym-log-cone} and the asymptotic expansion at infinity of the solution are tackled in Section \ref{sec:well-posedness}. Section \ref{sec:Pohozaev} is devoted to a Pohozaev-type identity which is derived for more general homogeneous anisoptropic $p$-Laplace equations with any $1<p<\infty$. Finally, we complete the proof of Theorem \ref{thm:over-log-cone} and of Proposition \ref{pro:observe} in Section \ref{sec:proof}.

\section{Preliminaries}\label{sec:prelimi}
\subsection{Basic properties of the function $H$}\label{subsec:pro-H}
Throughout this subsection, we always let $H:\mathbb{R}^N\to\mathbb{R}$ be a positively homogeneous function of degree one which is of class $C^0(\mathbb{R}^N\setminus\{0\})$ and satisfies $H(\xi)>0$ for any $\xi\in\mathbb{S}^{N-1}$, and let $H_0$ be its dual function defined by \eqref{intro-eq:def-H0}. Clearly, by definition, $H_0$ is also positively homogeneous of degree one and is convex. Here, recall that a function $f:\mathbb{R}^N\to\mathbb{R}$ is said to be positively homogeneous of degree one if
\begin{equation*} 
f(tz)=tf(z)\quad\text{for any } t>0,\,z\in\mathbb{R}^N.
\end{equation*}
From this, one infers that $H(0)=0$ and $H\in C^0(\mathbb{R}^N)$. Also, for any $\xi\in\mathbb{R}^N$,
$$|\xi|\min_{\mathbb{S}^{N-1}}H\leq H(\xi)\leq|\xi|\max_{\mathbb{S}^{N-1}}H.$$
It is clear that analogous properties hold with $H$ replaced by $H_0$.
Moreover, the homogeneity implies that
\begin{lemma}\label{pre-lem:pro-H}
If $H,H_0\in C^1(\mathbb{R}^N\setminus\{0\})$, then
$$\nabla H(t\xi)=\nabla H(\xi), \quad \nabla H_0(tx)=\nabla H_0(x)$$
for any $t>0$ and $\xi,x\in\mathbb{R}^N\setminus\{0\}$, and
$$\left<\nabla H(\xi),\xi\right>=H(\xi),\quad \left<\nabla H_0(x),x\right>=H_0(x)$$
for any $\xi,x\in\mathbb{R}^N$.
\end{lemma}

From Lemma \ref{pre-lem:pro-H}, one observes that $H^N\in C^1(\mathbb{R}^N)$ if $H\in C^1(\mathbb{R}^N\setminus\{0\})$, referring to \cite[Lemma 2.3]{Cozzi2014} for a rigorous proof. Thus, the function $a(\xi)= \frac{1}{N}\nabla H^N(\xi)$ is actually continuous at the origin and we have
\begin{equation}\label{intro-eq:def-a}
a(\xi)=
\begin{cases}
H^{N-1}(\xi)\nabla H(\xi), &\text{if }\xi\in\mathbb{R}^N\setminus\{0\},\\
0, & \text{if }\xi=0.
\end{cases}
\end{equation}

The next lemma collects several well-known properties containing the convexity of $H$, the differentiability of $H_0$ and useful connections between $H$ and $H_0$.


\begin{lemma}\label{pre-lem:pro-H-con}
If $H\in C^2(\mathbb{R}^N\setminus\{0\})$ and the Hessian of $H^N$ is positive definite in $\mathbb{R}^N\setminus\{0\}$, then $H$ is convex and $H_0\in C^2(\mathbb{R}^N\setminus\{0\})$. Moreover, for $x,\xi\in\mathbb{R}^N\setminus\{0\}$, 
\begin{equation}\label{pre-eq:pro-H-con-1}
H(\nabla H_0(x))=H_0(\nabla H(\xi))=1,
\end{equation} and 
\begin{equation}\label{pre-eq:pro-H-con-2}
x=H_0(x)\nabla H(\nabla H_0(x)),\quad \xi=H(
\xi)\nabla H_0(\nabla H(\xi)).
\end{equation}
\end{lemma}

\begin{proof}
Let us give the precise references for these assertions. The convexity of $H$ was proved in \cite[Lemma 2.5]{Cozzi2014}. The regularity that $H_0$ is of class $C^2$ outside the origin and the formula \eqref{pre-eq:pro-H-con-1} were obtained in \cite[Lemma 2.3]{Cozzi2016} (see also \cite[Lemma 3.1]{Cianchi-Salani-2009}). There the authors also stated that the map $H\nabla H$ is a $C^1$-diffeomorphism of $\mathbb{R}^N$ with the inverse $H_0\nabla H_0$, from which we get \eqref{pre-eq:pro-H-con-2}. 
\end{proof}
Related to the differentiability of $H_0$, we remark further that since $H_0$ is actually the support function of the set $B_1^H$ (given by \eqref{intro-eq:def-B1H}), it is known that $H_0\in C^1(\mathbb{R}^N\setminus\{0\})$ if and only if $B_1^H$ is strictly convex (see \cite[Corollary 1.7.3]{RS1993}). In addition, from \cite[Lemma 3.1]{Cianchi-Salani-2009}, formulas \eqref{pre-eq:pro-H-con-1} and \eqref{pre-eq:pro-H-con-2} hold as long as $H,H_0\in C^1(\mathbb{R}^N\setminus\{0\})$. 

We conclude this part by presenting the following estimates gained from assuming $H$ to be uniformly elliptic.

\begin{lemma}\label{pre-lem:H-Hessian}
Assume that $H\in C^2({\mathbb{R}^N\setminus\{0\}})$ is uniformly elliptic. Then there exists $\lambda>0$ such that
\begin{equation}\label{pre-eq:Hess-H}
\partial^2_{ij}H^N(\xi)\eta_i\eta_j\geq \frac{1}{\lambda}|\xi|^{N-2}|\eta|^2\quad\text{and}\quad
\sum_{i,j}|\partial^2_{ij}H^N(\xi)|\leq \lambda|\xi|^{N-2}
\end{equation}
for any $\xi\in\mathbb{R}^N\setminus\{0\}$, $\eta\in\mathbb{R}^N
$. Furthermore, there exist $c_1,c_2>0$, depending only on $N$ and $\lambda$, such that 
\begin{gather}
\left<a(\xi_1)-a(\xi_2),\xi_1-\xi_2\right>\geq c_1(|\xi_1|+|\xi_2|)^{N-2}|\xi_1-\xi_2|^2,\label{pre-eq:mo-H-1}\\
\left|a(\xi_1)-a(\xi_2)\right|\leq c_2(|\xi_1|+|\xi_2|)^{N-2}|\xi_1-\xi_2|\label{pre-eq:mo-H-2},
\end{gather}
for any $\xi_1,\xi_2\in\mathbb{R}^N\setminus\{0\}$.
\end{lemma}
\begin{proof}
We refer the validity of \eqref{pre-eq:Hess-H} to \cite[Theorem 1.5]{Cozzi2016}. Then estimates \eqref{pre-eq:mo-H-1} and \eqref{pre-eq:mo-H-2} follow by applying \cite[Lemma 2.1]{Damascelli1998}.
\end{proof}
 
Once \eqref{pre-eq:Hess-H} holds, the Finsler $N$-Laplacian $\Delta_N^H$ is a possibly degenerate elliptic operator and we are allowed to apply standard regularity theory for quasilinear PDEs developed in \cite{DiBen1983,Serrin1964,Tolksdorf1984}  to equation \eqref{eq:log-cone}. 

\subsection{Comparison principles} 
 We derive the following weak comparison principles for the operator $\Delta_N^H$ in bounded domains inside a convex cone. For convenience, we write below $\Gamma_0:=\Sigma\cap\partial E$ and $\Gamma_1:=\partial\Sigma\cap E$, consistent with the notation used in \eqref{eq:log-cone}. 
\begin{lemma}\label{pre-lem:WCP}
Let $\Sigma\subset\mathbb{R}^N$ be an open, convex cone and $E\subset\mathbb{R}^N$ be a bounded domain such that $\mathcal{H}^{N-1}(\Gamma_0)>0$ and $\Sigma\cap E$ is connected. Let $H$ be as in Lemma \ref{pre-lem:H-Hessian}. Assume that $u,v\in W^{1,N}(\Sigma\cap E)\cap C^0((\Sigma\cap E)\cup\Gamma_0)$ satisfy
\begin{equation}\label{pre-eq:WCP}
\left\{
\begin{array}{ll}
-\Delta_N^Hu\leq -\Delta_N^Hv &\text{in }\Sigma\cap E,\\
u\leq v &\text{on }\Gamma_0,\\
\left<a(\nabla u),\nu\right>=\left<a(\nabla v),\nu\right>=0 &\text{on }\Gamma_1.
\end{array}
\right.
\end{equation}
Then $u\leq v$ in $\Sigma\cap E$.
\end{lemma}

\begin{proof}
By the weak formulation of \eqref{pre-eq:WCP} and since $u\leq v$ on $\Gamma_0$, we can use $(u-v)^+$ as a test function to get
$$\int_{\Sigma\cap E}\left<a(\nabla u)-a(\nabla v),\nabla(u-v)^+\right>\,dx\leq 0.$$
Thus, it follows from \eqref{pre-eq:mo-H-1} that
$$\int_{\Sigma\cap E}|\nabla (u-v)^+|^{N}\,dx=0,$$
which means that $(u-v)^+$ is constant in $\Sigma\cap E$. Since $(u-v)^+=0$ on $\Gamma_0$, we infer that $(u-v)^+\equiv0$, i.e. $u\leq v$ in $\Sigma\cap E$.
\end{proof}

We remark that when $\Sigma=\mathbb{R}^N$ and $H$ is a norm, Lemma \ref{pre-lem:WCP} is well-known (see for instance \cite{Bianchini-Ciraolo-2018,Sciunzi2019} where more general anisotropic elliptic operators are concerned).

\subsection{Anisotropic isoperimetric inequality inside convex cones}
In this subsection, we shall let $H:\mathbb{R}^N\to\mathbb{R}$ be a gauge, i.e., a nonnegative, positively homogeneous of degree one, convex function; and we also assume that $H$ is positive on $\mathbb{S}^{N-1}$. Notice that if we further assume $H$ to be even, then it becomes a norm in $\mathbb{R}^N$.

Given an open subset $D\subset\mathbb{R}^N$ and a measurable set $E\subset\mathbb{R}^N$, we recall the definition of anisotropic perimeter of $E$ in $D$ with respect to the gauge $H$, given by
\begin{align}
P_H(E;D)&=\sup\left\{\int_E\mathrm{div}\,\Phi\,dx:\Phi\in C_0^1(D;\mathbb{R}^N),H_0(\Phi)\le1\right\}\label{pre-eq:def-a-perimeter}\\
&=\int_{D\cap\partial^*E}H(\nu)\,d\mathcal{H}^{N-1}\notag
\end{align}
where $H_0$ is the dual function of $H$ defined by \eqref{intro-eq:def-H0}, $\partial^*E$ is the reduced boundary of $E$ and $\nu$ is the outer normal to $E$. 
Then the following result holds.
\begin{theorem}\label{pre-thm:Wulff-cone}
Let $\Sigma=\mathbb{R}^k\times\tilde{\Sigma}$ where $k\in\{0,\cdots,N\}$ and $\tilde{\Sigma}\subset\mathbb{R}^{N-k}$ is an open, convex cone with vertex at the origin which contains no lines. Let $H$ be a gauge in $\mathbb{R}^N$ which is positive on $\mathbb{S}^{N-1}$ and let $H_0$ be its dual function defined by \eqref{intro-eq:def-H0}. Then for each measurable set $E\subset\mathbb{R}^N$ with $\mathcal{H}^N(\Sigma\cap E)<\infty$,
\begin{equation*}
\frac{P_{H}(E;\Sigma)}{\mathcal{H}^N(\Sigma\cap E)^{\frac{N-1}{N}}}\geq \frac{P_H(B_1^{H_0};\Sigma)}{\mathcal{H}^N(\Sigma\cap B_1^{H_0})^{\frac{N-1}{N}}}.
\end{equation*}
Moreover, the equality sign holds if and only if $\Sigma\cap E=\Sigma\cap B_R^{H_0}(x_0)$ for some $R>0$ and $x_0\in\mathbb{R}^k\times\{0\}$.
\end{theorem}
Here $B_R^{H_0}(x_0)$ is as in \eqref{intro-eq:def-BH0} and $B_1^{H_0}:=B_1^{H_0}(0)$.

\begin{proof}
As already mentioned in the Introduction, this inequality has been obtained in \cite[Theorem 1.3]{CRS2016}. The characterization of the equality cases was showed in \cite[Theorem 4.2]{DPV2021}  by adapting the ideas of \cite[Theorem 2.2]{Figalli-Indrei-2013}, under the  assumption that $H$ is a norm. Nevertheless, the argument in \cite{DPV2021} works the same way considering a positive gauge $H$, since in this case $\Sigma\cap B_1^{H_0}$ is still an  open bounded convex set. For this reason, we omit the proof.
\end{proof}

\subsection{Second-order regularity for weak solutions}\label{subsec:second-order}
This subsection is concerned with the regularity of $W_{loc}^{2,2}$ type for weak solutions $u$ to problem \eqref{eq:log-cone}. Inspired by the approach in \cite{Avelin2018,CM2018,CFR2020}, in Proposition \ref{pre-pro:second-order} below we establish a differentiability result about $\nabla u$, which will be useful in Section \ref{sec:Pohozaev} to derive a Pohozaev-type identity and in the proof of Theorem \ref{thm:over-log-cone} in Section \ref{sec:proof}. In particular, it allows us to avoid assuming $u$ is $C^1$ up to $\Gamma_1$ which is required in \cite{Ciraolo2020,DPV2021,Pacella-Tralli-2020} to prove symmetry results for analogous problems defined in convex cones.
\begin{proposition}\label{pre-pro:second-order}
Let $\Sigma\subset\mathbb{R}^N$ be an open, convex cone and $\Omega\subset\mathbb{R}^N$ be a bounded domain. Let $u$ be a weak solution to problem \eqref{eq:log-cone} where $H$ is as in Theorem \ref{thm:over-log-cone}. Then 
$a(\nabla u)\in W_{loc}^{1,2}((\Sigma\cap\overline{\Omega}^c)\cup\Gamma_1)$.
\end{proposition}

\begin{proof}
From \cite[Theorem 4.2]{Avelin2018} (see also \cite{ACF}), we already know that $a(\nabla u)\in W_{loc}^{1,2}(\Sigma\cap\overline{\Omega}^c)$. It suffices to show $a(\nabla u)\in W^{1,2}(B_\rho(x)\cap(\Sigma\cap\overline{\Omega}^c))$ for each $x\in\Gamma_1$ and some small $\rho=\rho(x)>0$. Hereafter, $B_\rho(x)$ denotes the Euclidean ball centered at $x$ and having radius $\rho$. The main issue is due to the degeneracy of the equation as well as the non-smoothness of $\Gamma_1$. We shall deal with that by approximation.

Fix $y\in\Gamma_1$ and $B_\rho(y)$ with small $\rho>0$ such that $B_\rho(y)$ is away from $\Omega$. Let $\Sigma_k\supset\Sigma$ be a family of smooth convex sets in $\mathbb{R}^N$ such that $\mathcal{H}^N((\Sigma_k\setminus\Sigma)\cap B_\rho(y))\to 0$ and the Hausdorff distance between $\Sigma_k$ and $\Sigma$ inside $B_\rho(y)$ tends to $0$ as $k\to\infty$. For simplicity, we will denote by $$B^\rho:=B_\rho(y)\cap\Sigma\quad\text{and}\quad B^{\rho,k}:=B_\rho(y)\cap\Sigma_k.$$
Also, as in \cite[Formula (4.58)]{CM2018} we require that the Lipschitz constant of $B^{\rho,k}$ satisfies
\begin{equation}\label{pre-eq:Lipschitz-Ome}
L_{B^{\rho,k}}\leq CL_{B^\rho}
\end{equation} 
for some constant $C$ independent of $k$.

Since $u\in W^{1,N}(B^\rho)$, by extension theorem, there is $\tilde{u}\in W^{1,N}(\mathbb{R}^N)$ such that $\tilde{u}=u$ in $B^\rho$.  Let us consider the following equation
\begin{equation}\label{pre-eq:appr-k-cone}
\left\{
\begin{array}{ll}
\Delta_N^H u_k=0 & \text{in }B^{\rho,k},\\
u_k=\tilde{u} & \text{on }\Sigma_k\cap\partial B_\rho(x):=\Gamma^0_{\rho,k},\\
\left<a(\nabla u_k),\nu\right>=0 & \text{on }\partial\Sigma_k\cap B_\rho(x):=\Gamma^1_{\rho,k}.
\end{array}
\right.
\end{equation}
We notice that the existence of a weak solution $u_k$ to \eqref{pre-eq:appr-k-cone} follows by solving the minimization problem
\begin{equation}\label{pre-eq:mini-u-k}
\inf\left\{\frac{1}{N}\int_{B^{\rho,k}}H^N(\nabla w)\,dx:w\in W^{1,N}(B^{\rho,k}), w=\tilde{u}\text{ on }\Gamma^0_{\rho,k}\right\}.
\end{equation}
We first claim that $u_k\to u$ in $C_{loc}^1(B^\rho)$, by taking, if necessary, a subsequence of $k\to\infty$. 

Indeed, choosing $u_k-\tilde{u}$ as a test function in weak formulation of \eqref{pre-eq:appr-k-cone} and applying the Poincar\'e inequality (see for instance \cite[Corollary 4.5.2]{Ziemer1989}), we find that $$\| u_k\|_{W^{1,N}(B^\rho)}\leq C\|u\|_{W^{1,N}(B^\rho)}$$ for some constant $C=C(N,H,B^\rho)$. Thus, the Sobolev embedding theorem ensures that, up to a subsequence,
\begin{equation*}
u_k\to v\text{ in } L^N(B^\rho)\quad\text{and}\quad u_k\rightharpoonup v\text{ in } W^{1,N}(B^\rho)
\end{equation*}
for some function $v\in W^{1,N}(B^\rho)$. Furthermore, by interior $L^\infty$ estimates in \cite{Serrin1964} and interior $C^{1,\gamma}$ estimates in \cite{DiBen1983,Tolksdorf1984} for quasilinear PDEs, we can infer that for every compact subset $K\subset B^\rho$, $\|u_k\|_{C^{1,\gamma}(K)}$ is uniformly bounded. Thus, by Arzela--Ascoli theorem we get
\begin{equation}\label{pre-eq:C1-app-u-k}
u_k\to v\quad\text{and}\quad\nabla u_k\to\nabla v \,\text{ pointwise in }B^\rho,
\end{equation}
along a subsequence of $k\to\infty$. These convergence results imply that the function $v$ is a weak solution to \eqref{pre-eq:appr-k-cone}. Therefore, $v=u$ by the uniqueness.

Next, we will show $a(\nabla u_k)\in W^{1,2}(B^{\frac{\rho}{6},k})$ and derive a uniform bound for $\|a(\nabla u_k)\|_{W^{1,2}(B^{\frac{\rho}{6},k})}$. The idea is to introduce a family of regularized non-degenerate equations \eqref{pre-eq:appr-e-k-cone} below and establish a Caccioppoli type estimate for their solutions $u_k^\epsilon$ which approximate $u_k$. Then the desired bound is obtained by taking the limit as $\epsilon\to 0$. 

We start by letting $\{\phi_\epsilon\}$ with $\epsilon\in(0,1)$ be a family of radially symmetric smooth mollifiers and define $$a^\epsilon(\xi):=(a*\phi_\epsilon)(\xi)\quad\text{for }\xi\in\mathbb{R}^N.$$
Here the symbol $*$ stands for the convolution. Since $a(\cdot)$ is continuous, standard properties of convolutions imply $a^\epsilon\to a$ uniformly on compact subset of $\mathbb{R}^N$, as $\epsilon\to 0$. Moreover, following \cite[Lemma 2.4]{FF1997} it is seen that $a^\epsilon$ satisfies
\begin{equation*}
\left<\nabla a^\epsilon(\xi)\eta,\eta\right>\geq\frac{1}{\lambda}(|\xi|^2+\epsilon^2)^{\frac{N-2}{2}}|\eta|^2\quad\text{and}\quad|\nabla a^\epsilon(\xi)|\leq\lambda(|\xi|^2+\epsilon^2)^{\frac{N-2}{2}}
\end{equation*}
for every $\xi,\eta\in\mathbb{R}^N$, with $\lambda>0$ given by \eqref{pre-eq:Hess-H}. In addition, from \cite[Formula (2.4)]{Avelin2018}, the following monotonicity condition holds:
\begin{equation}\label{pre-eq:monoto-e}
\left<a^\epsilon(\xi)-a^\epsilon(\eta),\xi-\eta\right>\geq C(N,\lambda)(|\xi|^2+|\eta|^2+\epsilon^2)^{\frac{N-2}{2}}|\xi-\eta|^2.
\end{equation}

Let $u^\epsilon_k\in W^{1,N}(B^{\rho,k})$ be the weak solution of 
\begin{equation}\label{pre-eq:appr-e-k-cone}
\left\{
\begin{array}{ll}
\mathrm{div}\,(a^\epsilon(\nabla u^\epsilon_k))=0 & \text{in }B^{\rho,k},\\
u^\epsilon_k=\tilde{u} & \text{on }\Gamma^0_{\rho,k},\\
\left<a^\epsilon(\nabla u^\epsilon_k),\nu\right>=0 & \text{on }\Gamma^1_{\rho,k},
\end{array}
\right.
\end{equation}
which can be obtained by considering \eqref{pre-eq:mini-u-k} where the function $H^N$ is replaced by $H^N_\epsilon$ defined as $$H^N_\epsilon(\xi):=(H^N*\phi_\epsilon)(\xi)\quad\text{for }\xi\in\mathbb{R}^N.$$ 
Since equation \eqref{pre-eq:appr-e-k-cone} is non-degenerate and $\Gamma^1_{\rho,k}$ is smooth, one has $u^\epsilon_k\in C_{loc}^2(B^{\rho,k}\cup\Gamma^1_{\rho,k})$ by classical regularity theory for elliptic equations.

 We also let $\varphi\in C_0^\infty(B_\rho(y))$ and let $\zeta_\delta:\overline{B^{\rho,k}}\to[0,1]$ be a family of smooth functions\footnote{Such $\zeta_\delta$ can be constructed as follows. For $\delta>0$, let $\psi_\delta\in C^\infty(\mathbb{R})$ with $\psi_\delta=0$ on $(-\infty,\delta]$, $\psi_\delta=1$ on $[2\delta,+\infty)$ and $\psi'_\delta\leq\frac{2}{\delta}$. Since $\Gamma^1_{\rho,k}$ is smooth, the distance function $\mathrm{dist}(x,\Gamma^1_{\rho,k})$ is also smooth on $$\{x\in\overline{B^{\rho,k}}:\mathrm{dist}(x,\Gamma^1_{\rho,k})<3\delta\},$$
provided $\delta$ is sufficiently small (see \cite[Lemma 14.16]{GT1998}). Set $\zeta_\delta=\psi_\delta(\mathrm{dist}(x,\Gamma^1_{\rho,k}))$.} such that $\zeta_\delta\to\chi_{B^{\rho,k}}$ in the $L^1$ sense and $-\nabla\zeta_\delta\to\nu\mathcal{H}^{N-1}\llcorner\partial B^{\rho,k}$ in the sense of measures, as $\delta\to 0$, where $\nu$ is the outer normal to $\partial B^{\rho,k}$. 

Now, for $m\in\{1,\cdots,N\}$, using $\partial_m(\varphi\zeta_\delta)$ as the test function in the weak formulation of \eqref{pre-eq:appr-e-k-cone}, we get
$$\sum_{i=1}^N\left(\int_{B^{\rho,k}}\zeta_\delta\partial_ma^\epsilon_i(\nabla u^\epsilon_k)\partial_i\varphi\,dx+\int_{B^{\rho,k}}\varphi\partial_ma^\epsilon_i(\nabla u^\epsilon_k)\partial_i\zeta_\delta\,dx\right)=0,$$
where we used the notation $a^\epsilon=(a^\epsilon_1,\cdots,a^\epsilon_N)$ to denote the components of the vector field $a^\epsilon$. Thus, letting $\delta\to 0$ in the above equality yields
\begin{equation}\label{pre-eq:CFR}
\sum_{i=1}^N\left(\int_{B^{\rho,k}}\partial_ma^\epsilon_i(\nabla u^\epsilon_k)\partial_i\varphi\,dx-\int_{\Gamma^1_{\rho,k}}\varphi\partial_ma^\epsilon_i(\nabla u^\epsilon_k)\nu_i\,d\mathcal{H}^{N-1}\right)=0.
\end{equation}
By density, in \eqref{pre-eq:CFR} we actually can choose any $\varphi\in W^{1,2}(B^{\rho,k})$ with $\varphi=0$ on $\Gamma_{\rho,k}^0$. In particular, let us take $\varphi=a^\epsilon_m(\nabla u^\epsilon_k)\psi^2$, where $\psi\in C_0^\infty(B_\rho(y))$. For this choice, using the fact that $\partial\Sigma_k$ is convex and arguing as in the proof of \cite[Proposition 2.8]{CFR2020} (see Formulas (2.45)--(2.50) there), we deduce that
\begin{equation*}
\sum_{i,m=1}^N\int_{B^{\rho,k}}\partial_ma^\epsilon_i(\nabla u^\epsilon_k)\partial_i\left(a^\epsilon_m(\nabla u^\epsilon_k)\psi^2\right)dx\leq0.
\end{equation*}
Furthermore, we can argue as in the proof of \cite[Theorem 4.2]{Avelin2018} to obtain the following Caccioppoli type estimate
\begin{equation}\label{pre-eq:Cacci}
\int_{B^{\rho,k}}|\nabla(a^\epsilon(\nabla u^\epsilon_k))|^2\psi^2\,dx\leq C\int_{B^{\rho,k}}|a^\epsilon(\nabla u^\epsilon_k)|^2|\nabla\psi|^2\,dx
\end{equation}
for some constant $C=C(N,\lambda)$.

 Pick a $\psi\in C_0^\infty(B_{\frac{\rho}{5}}(y))$ such that $\psi=1$ in $B_{\frac{\rho}{6}}(y)$ and $|\nabla\psi|\leq\frac{C}{\rho}$ for some constant $C=C(N)$. It follows from \eqref{pre-eq:Cacci} that
\begin{equation}\label{pre-eq:Cacci-psi}
\|a^\epsilon(\nabla u^\epsilon_k)\|^2_{W^{1,2}(B^{\frac{\rho}{6},k})}\leq\frac{C(N,\lambda)}{\rho^2}\int_{B^{\frac{\rho}{5},k}}|a^\epsilon(\nabla u^\epsilon_k)|^2\,dx.
\end{equation}
In order to estimate the right-hand side, we fix a small $\tau\in(0,\frac{\rho}{20})$ and assume that for some $y_k\in\Gamma^1_{\rho,k}$, $|y-y_k|=\mathrm{dist}(y,\Gamma^1_{\rho,k})<\tau$ when $k$ is sufficiently large. It is seen that for each large $k$,
$$B_{\frac{\rho}{5}}(y)\subset B_{\frac{\rho}{4}}(y_k)\subset B_{\frac{\rho}{2}}(y_k)\subset\subset B_{\rho}(y).$$
Thus, by virtue of a local flattening argument for $\Gamma^1_{\rho,k}$ and a $L^\infty$ estimate for $\nabla u^\epsilon_k$ near the boundary (see Formula (4.4) in \cite[Proof of Lemma 6]{Lieberman1988}), it is not difficult to deduce that 
\begin{equation}\label{pre-eq:bound-Gradient}
\|\nabla u^\epsilon_k\|_{L^\infty\left(B^{\frac{\rho}{5},k}\right)}\leq\|\nabla u^\epsilon_k\|_{L^\infty\left(B_{\frac{\rho}{4}}(y_k)\cap\Sigma_k\right)}\leq C
\end{equation}
for some constant $C=C(N,\lambda,\|\nabla u^\epsilon_k\|_{L^N(B^{\rho,k})},\rho,L_{B^{\rho,k}})$. Consequently, combining \eqref{pre-eq:Cacci-psi} with the relation $|a^\epsilon(\xi)|\leq\lambda(|\xi|+\epsilon)^{N-1}$ and with \eqref{pre-eq:bound-Gradient} gives
\begin{equation}\label{pre-eq:Cacci-e}
\|a^\epsilon(\nabla u^\epsilon_k)\|_{W^{1,2}(B^{\frac{\rho}{6},k})}\leq C(N,\lambda,\|\nabla u^\epsilon_k\|_{L^N(B^{\rho,k})},\rho,L_{B^{\rho,k}}).
\end{equation}
Moreover, choosing $u^\epsilon_k-\tilde{u}$ as the test function in the weak formulation of \eqref{pre-eq:appr-e-k-cone}, it follows from \eqref{pre-eq:monoto-e} that  $$\|\nabla u^\epsilon_k\|_{L^N(B^{\rho,k})}\leq C\|\nabla\tilde{u}\|_{L^N(B^{\rho,k})}$$ for some constant $C=C(N,\lambda,\mathcal{H}^{N}(B^{\rho,k}))$. In view of \eqref{pre-eq:Cacci-e}, we hence arrive at
\begin{equation}\label{pre-eq:Cacci-uniform}
\|a^\epsilon(\nabla u^\epsilon_k)\|_{W^{1,2}(B^{\frac{\rho}{6},k})}\leq C(N,\lambda,\|\nabla\tilde{u}\|_{L^N(B^{\rho,k})},\rho,L_{B^{\rho,k}}).
\end{equation}

At this point, to obtain an estimate for $\|a(\nabla u_k)\|_{W^{1,2}(B^{\frac{\rho}{6},k})}$, we present certain convergence results and then let $\epsilon\to 0$ in \eqref{pre-eq:Cacci-uniform}. Exploiting condition \eqref{pre-eq:monoto-e} and arguing as the derivation of Formula (4.2) contained in \cite[Proof of Theorem 4.2]{Avelin2018}, we easily get $$\nabla u^\epsilon_k\to\nabla u_k\quad\text{in }L^N(B^{\rho,k}),\quad\text{as }\epsilon\to 0,$$ which implies $a^\epsilon(\nabla u^\epsilon_k)\to a(\nabla u_k)$ in $L^1(B^{\rho,k})$. Moreover, since the constant $C$ in \eqref{pre-eq:Cacci-uniform} is independent of $\epsilon$, we have 
\begin{equation*}
a^\epsilon(\nabla u^\epsilon_k)\rightharpoonup a(\nabla u_k)\quad\text{in }W^{1,2}(B^{\frac{\rho}{6},k}),
\end{equation*}
for a subsequence of ${\epsilon}\to 0$.
Therefore, the lower semicontinuity
for weak convergence leads to 
\begin{equation}
\label{pre-eq:Cacci-uniform-u-k}
\|a(\nabla u_k)\|_{W^{1,2}(B^{\frac{\rho}{6},k})}\leq C(N,\lambda,\|\nabla\tilde{u}\|_{L^N(B^{\rho,k})},\rho,L_{B^{\rho,k}}).
\end{equation}

Finally, with \eqref{pre-eq:Cacci-uniform-u-k} in hand, let us prove $a(\nabla u)\in W^{1,2}(B^{\frac{\rho}{6}})$. In view of \eqref{pre-eq:Lipschitz-Ome} and the fact that $\|\nabla\tilde{u}\|_{L^N(B^{\rho,k})}\leq C$ for some $C$ not depending on $k$, we observe that the constant $C$ in \eqref{pre-eq:Cacci-uniform-u-k} is actually independent of $k$. Hence, $a(\nabla u_k)$ is uniformly bounded in $W^{1,2}(B^{\frac{\rho}{6}})$ and there exists a function $U\in W^{1,2}(B^{\frac{\rho}{6}})$ such that, up to a subsequence of $k\to\infty$,
\begin{equation}\label{pre-eq:W-app-u-k}
a(\nabla u_k)\to U\text{ in }L^2(B^{\frac{\rho}{6}})\quad\text{and}\quad a(\nabla u_k)\rightharpoonup U\text{ in }W^{1,2}(B^{\frac{\rho}{6}}).
\end{equation}
Via \eqref{pre-eq:W-app-u-k} and \eqref{pre-eq:C1-app-u-k}, we infer that
$$U=a(\nabla u)\in W^{1,2}(B^{\frac{\rho}{6}}).$$

This completes the proof.
\end{proof}

\section{Asymptotic expansion and the solvability}\label{sec:well-posedness}

This section is devoted to the study of problem \eqref{eq:log-cone}--\eqref{eq:asym-log-cone}. We first show that the prescribed logarithmic behavior \eqref{eq:asym-log-cone} can be improved to a precise asymptotic expansion near infinity, by using scaling arguments and comparison principle together with a Liouville-type result (Lemma \ref{cap-lem:Liouville} below). Thanks to such improved asymptotics, we prove the existence and uniqueness of weak solutions to problem \eqref{eq:log-cone}--\eqref{eq:asym-log-cone}.

\begin{proposition}\label{cap-pro:asym-equi}
Let $\Sigma$, $H$ and $H_0$ be as in Theorem \ref{thm:over-log-cone} and let $\Omega\subset\mathbb{R}^N$ be a bounded domain. Let $u$ be a weak solution to problem \eqref{eq:log-cone}--\eqref{eq:asym-log-cone}. Then there exist $\gamma>0$ and $\beta\in\mathbb{R}$ such that
\begin{equation}\label{cap-eq:asym-equi}
\lim_{H_0(x)\to\infty}\left(u(x)-\gamma\ln H_0(x)\right)=\beta,
\end{equation}
and 
\begin{equation}\label{cap-eq:asym-D-equi}
\lim_{H_0(x)\to\infty} H_0(x)H(\nabla (u-\gamma\ln H_0(x)))=0.
\end{equation}
\end{proposition}

\begin{proof}
Let $R_0$ be such that $\Omega\subset B_{R_0}^{H_0}:=B_{R_0}^{H_0}(0)$ and set $$\gamma:=\limsup_{H_0(x)\to\infty}\frac{u(x)}{\ln H_0(x)}.$$ We first show 
\begin{equation}\label{cap-eq:prop-bounded}
u(x)-\gamma\ln H_0(x)\in L^\infty(\Sigma\setminus B_{R_0}^{H_0}).
\end{equation}

Let 
$$u_R(x)=\frac{u(Rx)}{\ln R},\quad\text{for }x\in\Sigma\setminus B^{H_0}_{\frac{R_0}{R}}.$$
Then 
\begin{equation*}
\left\{
\begin{array}{ll}
\Delta_N^H u_R=0 &\text{in }\Sigma\setminus B^{H_0}_{\frac{R_0}{R}},\\
\left<a(\nabla u_R),\nu\right>=0 &\text{on }\partial\Sigma\setminus B^{H_0}_{\frac{R_0}{R}}.
\end{array}
\right.
\end{equation*}
From the condition \eqref{eq:asym-log-cone} and by using $1$-homogeneity of $H_0$, we have $$|u_R(x)|\leq d\left(1+\frac{|\ln H_0(x)|}{\ln R}\right),$$
and in particular $u_R(x)$ is bounded in every compact subset of $\overline{\Sigma}\setminus\{0\}$, uniformly in $R$. From regularity theory for quasilinear PDEs \cite{DiBen1983,Serrin1964,Tolksdorf1984}, we deduce that $u_R(x)$ is uniformly bounded in $C^{1,\alpha}_{loc}(\Sigma)\cap W_{loc}^{1,N}(\overline{\Sigma}\setminus\{0\})$\footnote{Here and in the following, to avoid a misunderstanding, we  point out that when $\Sigma=\mathbb{R}^N$ the regularity of $C_{loc}^{1,\alpha}(\Sigma)$ needs to be replaced by $C_{loc}^{1,\alpha}(\mathbb{R}^N\setminus\{0\})$.} with respect to $R$. Hence, by Arzela--Ascoli theorem and a diagonal process, $u_{R_j}\to v$ in $C^1_{loc}(\Sigma)$ and $u_{R_j}\rightharpoonup v$ in $W_{loc}^{1,N}(\overline{\Sigma}\setminus\{0\})$ along a sequence $R_j\to\infty$, where $v\in W_{loc}^{1,N}(\overline{\Sigma}\setminus\{0\})\cap L^\infty(\Sigma)$ satisfies 
\begin{equation*}
\left\{
\begin{array}{ll}
\Delta_N^Hv=0 &\text{in }\Sigma,\\
\left<a(\nabla v),\nu\right>=0 &\text{on }\partial\Sigma\setminus\{0\}.
\end{array}
\right.
\end{equation*}
By Lemma \ref{cap-lem:Liouville} below, $v$ is constant.

We claim that $v=\gamma$ by proving that
\begin{equation}\label{cap-eq:lim-gamma}
\lim_{H_0(x)\to\infty}\frac{u(x)}{\ln H_0(x)}=v.
\end{equation}
Indeed, for $\epsilon>0$, there exists $j(\epsilon)\in\mathbb{N}$ such that $$(v-\epsilon)\ln H_0(R_jx)\leq u(R_jx)\leq (v+\epsilon)\ln H_0(R_jx)$$ for $j\geq j(\epsilon)$ and $x\in\partial B^{H_0}_1$. Since $\left<a(\nabla\ln H_0(x)),\nu\right>=\left<x,\nu\right>=0$ a.e. on $\partial\Sigma$ by formula \eqref{pre-eq:pro-H-con-2}, by applying Lemma \ref{pre-lem:WCP} to $u(z)$ and $\ln H_0(z)$ we thus obtain
$$(v-\epsilon)\ln H_0(z)\leq u(z)\leq (v+\epsilon)\ln H_0(z)$$
for any $z\in\Sigma$ such that $H_0(z)\geq R_{j(\epsilon)}$. This implies \eqref{cap-eq:lim-gamma} and hence $v=\gamma$.

For $\epsilon>0$, let 
\begin{gather*}
\bar{u}_\epsilon(x)=(\gamma+\epsilon)\ln H_0(x)-(\gamma+\epsilon)\ln R_0+\sup_{\partial B^{H_0}_{R_0}}u,\\
\underline{u}_\epsilon(x)=(\gamma-\epsilon)\ln H_0(x)-(\gamma-\epsilon)\ln R_0+\inf_{\partial B^{H_0}_{R_0}}u.
\end{gather*}
Then, $$\underline{u}_\epsilon\leq u\leq\bar{u}_\epsilon\quad\text{on }\partial B^{H_0}_{R_0}\text{ and also for }H_0(x)\text{ large enough}.$$ 
Consequently, by Lemma \ref{pre-lem:WCP} again and letting $\epsilon\to 0$, we deduce that $$\inf_{\partial B^{H_0}_{R_0}}u-\gamma\ln R_0\leq u-\gamma\ln H_0(x)\leq\sup_{\partial B^{H_0}_{R_0}}u-\gamma\ln R_0$$
in $\Sigma\setminus B^{H_0}_{R_0}$, which implies \eqref{cap-eq:prop-bounded}.

Now, we prove the asymptotic behaviors at infinity of $u$ and $\nabla u$.

For $m>0$, we introduce the function $$u_m(y):=u(my)-\gamma\ln m.$$
By setting $G(x):=u(x)-\gamma\ln H_0(x)$, we also have that $$u_m(y)=\gamma\ln H_0(y)+G(my).$$ Since $G(x)\in L^\infty(\Sigma\setminus B_{R_0}^{H_0})$, $u_m(y)$ is bounded in every compact subset of $\overline{\Sigma}\setminus\{0\}$, uniformly in $m$. Similarly as done for $u_R$ above, we have that $u_m(y)$ is uniformly bounded in $C^{1,\alpha}_{loc}(\Sigma)\cap W_{loc}^{1,N}(\overline{\Sigma}\setminus\{0\})$ with respect to $m$. Consequently, there is a sequence $m_j\to\infty$ such that $$u_{m_j}\to u_\infty\text{ in } C^1_{loc}(\Sigma)\text{ and }u_{m_j}\rightharpoonup u_\infty \text{ in } W_{loc}^{1,N}(\overline{\Sigma}\setminus\{0\})$$ where $u_\infty$ satisfies 
\begin{equation*}
\left\{
\begin{array}{ll}
\Delta_N^Hu_\infty=0 &\text{in }\Sigma,\\
\left<a(\nabla u_\infty),\nu\right>=0 &\text{on }\partial\Sigma\setminus\{0\}.
\end{array}
\right.
\end{equation*}
 If we set $$G_\infty(y):=u_\infty(y)-\gamma\ln H_0(y),$$ then $G_\infty(y)\in L^\infty(\Sigma)$. By applying Lemma \ref{cap-lem:Liouville}, we thus infer that $$G_\infty\equiv\beta$$ for some constant $\beta\in\mathbb{R}$. This implies that $$\lim_{m_j\to\infty}\left(u(m_jy)-\gamma\ln H_0(m_jy)\right)=\beta$$ in the $C^1_{loc}(\Sigma)$ topology. Via Lemma \ref{pre-lem:WCP}, we get \eqref{cap-eq:asym-equi}.  

Moreover, setting $G_m(y):=G(my)$, we have
\begin{equation}\label{cap-eq:D-G}
\sup_{H_0(x)=m}H_0(x)\left|\nabla\left(u(x)-\gamma\ln H_0(x)\right)\right|=\sup_{H_0(y)=1}|\nabla G_m(y)|.
\end{equation}
Since $G_{m_j}\to G_\infty$ in $C^1_{loc}(\Sigma)$, then
\begin{equation}\label{cap-eq:asym-D-G}
\sup_{H_0(y)=1}|\nabla G_{m_j}(y)|\to \sup_{H_0(y)=1}|\nabla G_\infty(y)|=0.
\end{equation}
It follows from \eqref{cap-eq:D-G}--\eqref{cap-eq:asym-D-G} that 
$$\sup_{H_0(x)=m}H_0(x)\left|\nabla\left(u(x)-\gamma\ln H_0(x)\right)\right|\to 0$$
holds for any sequence $m\to\infty$ up to extracting a subsequence. This implies
the validity of \eqref{cap-eq:asym-D-equi}, thus completing the proof.
\end{proof}

In the proof of Proposition \ref{cap-pro:asym-equi}, we have used the following rigidity result of Liouville-type. 
\begin{lemma}\label{cap-lem:Liouville}
Let $\Sigma$, $H$ and $H_0$ be as in Theorem \ref{thm:over-log-cone}. Let $\gamma\in\mathbb{R}$ be a constant. Assume that $G(x)\in W_{loc}^{1,N}(\overline{\Sigma}\setminus\{0\})\cap L^\infty({\Sigma})$ and the function $\gamma\ln H_0(x)+G(x)$ satisfies 
\begin{equation*}
\left\{
\begin{array}{ll}
\Delta_N^H\left(\gamma\ln H_0(x)+G(x)\right)=0 &\text{in }\Sigma,\\
\left<a\left(\nabla (\gamma\ln H_0(x)+G(x))\right),\nu\right>=0 &\text{on }\partial\Sigma\setminus\{0\}.
\end{array}
\right.
\end{equation*}
Then $G(x)$ is a constant function.
\end{lemma}

If $\Sigma=\mathbb{R}^N$, the homogeneous Neumann boundary condition above is trivially satisfied and this result has been shown in the Euclidean case (i.e. $H$ is the Euclidean norm), referring to \cite[Theorem 2.2]{K-Veron-1986} (see also \cite[Lemma 4.3]{Esposito2018} for an alternative proof). Since the argument in \cite{K-Veron-1986} relies on the Kelvin transform, which is, however, not helpful neither for cones nor for anisotropic equations, we shall prove Lemma \ref{cap-lem:Liouville} by adapting the one given in \cite{Esposito2018} to the conical-anisotropic setting, which is based on appropriate cut-off functions.

\begin{proof}[Proof of Lemma \ref{cap-lem:Liouville}]
For simplicity, we set $$A(x):=\gamma\ln H_0(x)+G(x).$$ It holds that 
\begin{equation}\label{cap-eq:test-A}
-\Delta_N^H A+\Delta_N^H(\gamma\ln H_0(x))=0\quad\text{in }\Sigma.
\end{equation}
Let $\eta$ be a cut-off with compact support in $\mathbb{R}^N\setminus\{0\}$. Testing \eqref{cap-eq:test-A} with $\eta^NG$ yields
\begin{align*}
I&:=\int_{\Sigma}\eta^N\left<a(\nabla A)-a(\nabla(\gamma\ln H_0(x))),\nabla G\right>dx\\
&=-N\int_{\Sigma}\eta^{N-1}G\left<a(\nabla A)-a(\nabla(\gamma\ln H_0(x))),\nabla\eta\right>dx:=II,
\end{align*}
where we used the fact that $$\left<a(\nabla A),\nu\right>=\left<a(\nabla(\gamma\ln H_0(x))),\nu\right>=0,\quad\text{on }\partial\Sigma\setminus\{0\}.$$
By \eqref{pre-eq:mo-H-1} and \eqref{pre-eq:mo-H-2}, we get that
\begin{equation}\label{cap-eq:lower-I}
I\geq c_1\int_\Sigma \eta^N|\nabla G|^N\,dx,
\end{equation}
and 
\begin{align}
II&\leq CN\|G\|_{L^\infty}\int_\Sigma\eta^{N-1}\left(|\nabla G|^{N-1}+\frac{|\nabla G|}{|x|^{N-2}}\right)|\nabla\eta|\,dx\notag\\
&\leq\frac{c_1}{2}\int_\Sigma\eta^N|\nabla G|^N\,dx+C\int_\Sigma\left(|\nabla\eta|^N+\frac{|\nabla\eta|^{\frac{N}{N-1}}}{|x|^{\frac{N(N-2)}{N-1}}}\right)dx.\label{cap-eq:upper-II}
\end{align}
By combining \eqref{cap-eq:lower-I} and \eqref{cap-eq:upper-II}, we thus obtain 
\begin{equation}\label{cap-eq:I+II}
\int_\Sigma\eta^N|\nabla G|^N\,dx\leq C\int_\Sigma\left(|\nabla\eta|^N+\frac{|\nabla\eta|^{\frac{N}{N-1}}}{|x|^{\frac{N(N-2)}{N-1}}}\right)dx.
\end{equation}

For $0<\delta<1$, we now choose $\eta$ as
\begin{equation*}
\eta(x)=\left\{
\begin{array}{ll}
0 &\text{if }|x|\leq \delta^2,\\
-\frac{\ln|x|-2\ln\delta}{\ln\delta} &\text{if }\delta^2\leq|x|\leq\delta,\\
1 &\text{if }\delta\leq|x|\leq\frac{1}{\delta},\\
\frac{\ln|x|+2\ln\delta}{\ln\delta} &\text{if }\frac{1}{\delta}\leq|x|\leq\frac{1}{\delta^2},\\
0 &\text{if }|x|\geq\frac{1}{\delta^2}.
\end{array}
\right.
\end{equation*}
After a direct computation, one has that
$$\int_{\mathbb{R}^N}|\nabla\eta|^N\,dx=\frac{C_N}{|\ln\delta|^{N-1}}\to 0$$
and $$\int_{\mathbb{R}^N}\frac{|\nabla\eta|^{\frac{N}{N-1}}}{|x|^{\frac{N(N-2)}{N-1}}}\,dx=\frac{C_N}{|\ln\delta|^{\frac{1}{N-1}}}\to 0$$
as $\delta\to0$, where $C_N:=\mathcal{H}^{N-1}(\mathbb{S}^{N-1})$. Consequently, letting $\delta\to0$ in \eqref{cap-eq:I+II}, we deduce that $$\int_\Sigma |\nabla G|^N\,dx=0.$$ This implies that $G$ is constant.
\end{proof}

Now let us investigate the solvability of problem \eqref{eq:log-cone}--\eqref{eq:asym-log-cone}. We establish the following existence and uniqueness theorem for weak solutions.

\begin{theorem}\label{cap-thm:F-log}
Let $\Sigma$, $\Omega$, $H$, $H_0$ and $x_0$ be the same as in Theorem \ref{thm:over-log-cone}. Assume $x_0\in\Omega$. Problem \eqref{eq:log-cone}--\eqref{eq:asym-log-cone} admits a unique weak solution $u$ fulfilling
\begin{equation}\label{cap-eq:thm-asym}
\lim_{H_0(x)\to\infty}\left(u(x)-\ln H_0(x)\right)=\beta
\end{equation}
for some constant $\beta\in\mathbb{R}$.
\end{theorem}

\begin{proof}
Without loss of generality, we prove the assertion in the case that $x_0$ is the origin.

Let $R>1$ be such that $B^{H_0}_R\supset\overline{\Omega}$. We first solve the following local problem
 \begin{equation}\label{cap-eq:F-equi-loc}
\left\{
\begin{array}{ll}
\Delta^H_N U_R=0 & \text{in }\Sigma\cap(B^{H_0}_R\setminus \overline{\Omega}),\\
U_R=0 & \text{on }\Sigma\cap\partial\Omega:=\Gamma_0,\\
\left<a(U_R),\nu\right>=0 &\text{on }\partial\Sigma\cap(B^{H_0}_R\setminus\overline{\Omega}),\\
U_R=\ln R & \text{on }\Sigma\cap\partial B^{H_0}_R:=\Gamma_R.
\end{array}
\right.
\end{equation}
To this aim, we consider the minimization problem
\begin{equation}\label{cap-eq:minimizer}
\inf\left\{\int_{\Sigma\cap B^{H_0}_R}H^N(\nabla\varphi)\,dx:\varphi\in V\right\}
\end{equation}
where 
\begin{align*}
V:=&\{\varphi\in W^{1,N}(\Sigma\cap B^{H_0}_R):\varphi=0\text{ on }\Gamma_R\text{ and }\\&\varphi-f=w\chi_{\Sigma}\text{ for some }w\in W^{1,N}_0(B^{H_0}_R\setminus\overline{\Omega})\}.
\end{align*}
Here $f\in C_0^\infty(B^{H_0}_R)$ is a given function such that $f=1$ in a neighborhood of $\Omega$. 

By the Poincar\'e inequality and a standard variational argument (see for example \cite{Lewis1977}),
there is a minimizer $v_R$ to \eqref{cap-eq:minimizer} which solves the Euler-Lagrange equation
 \begin{equation*}
\left\{
\begin{array}{ll}
\Delta^H_N v_R=0 & \text{in }\Sigma\cap(B^{H_0}_R\setminus \overline{\Omega}),\\
v_R=1 & \text{on }\Gamma_0,\\
\left<a(v_R),\nu\right>=0 &\text{on }\partial\Sigma\cap(B^{H_0}_R\setminus\overline{\Omega}),\\
v_R=0 & \text{on }\Gamma_R.
\end{array}
\right.
\end{equation*}
Thus, we obtain that $$U_R(x)=(1-v_R(x))\ln R$$
is the unique solution of \eqref{cap-eq:F-equi-loc}.

We next show that the approximation of $U_R$ is a solution to problem \eqref{eq:log-cone}--\eqref{eq:asym-log-cone}.

 Fix $R_1>0$ such that $B^{H_0}_{R_1}\subset\Omega$ and let $U_{R_1,R}$ be the solution to \eqref{cap-eq:F-equi-loc} for $\Omega=B^{H_0}_{R_1}$. By Lemma \ref{pre-lem:WCP} it holds
\begin{equation}\label{cap-eq:comparison-loc}
0\leq U_R(x)\leq U_{R_1,R} \quad\text{for }x\in\Sigma\cap (B^{H_0}_R\setminus\overline{\Omega}),
\end{equation}
where $$U_{R_1,R}=\frac{\ln R\left(\ln H_0(x)-\ln R_1\right)}{\ln R-\ln R_1}.$$
Since $U_{R_1,R}$ is uniformly bounded in $L_{loc}^\infty(\overline{\Sigma}\cap\Omega^c)$ with respect to $R$, so is the function $U_R$. Hence, applying classical regularity results for quasilinear PDEs \cite{DiBen1983,Serrin1964,Tolksdorf1984} we deduce that $U_R$ is bounded in $C_{loc}^{1,\alpha}(\Sigma\cap\overline{\Omega}^c)\cap W_{loc}^{1,N}(\overline{\Sigma}\cap\Omega^c)$, uniformly in $R$. By the Arzela--Ascoli theorem and a diagonal process one can find a sequence $R_j\to\infty$ so that $U_{R_j}\to u$ in $C_{loc}^1(\Sigma\cap\overline{\Omega}^c)$ and $U_{R_j}\rightharpoonup u$ in $W_{loc}^{1,N}(\overline{\Sigma}\cap\Omega^c)$. Moreover, thanks to the boundary regularity result in \cite{Lieberman1988}, it follows that $u\in C^{1,\alpha}((\Sigma\cap\overline{\Omega}^c)\cup\Gamma_0)$. With these assertions, it turns out that $u$ is a solution of problem \eqref{eq:log-cone}. It suffices to verify that $u$ satisfies \eqref{eq:asym-log-cone}.

Indeed, fix $R_2>0$ such that $B^{H_0}_{R_2}\supset\Omega$ and let $$U_{R_2,R}=\frac{\ln R\left(\ln H_0(x)-\ln R_2\right)}{\ln R-\ln R_2}$$ be the solution to \eqref{cap-eq:F-equi-loc} for $\Omega=B^{H_0}_{R_2}$. It holds
\begin{equation}\label{cap-eq:comparison-loc-lower}
U_R(x)\geq U_{R_2,R} \quad\text{for }x\in\Sigma\cap (B^{H_0}_R\setminus B^{H_0}_{R_2}).
\end{equation}
Passing to the limit in \eqref{cap-eq:comparison-loc} and \eqref{cap-eq:comparison-loc-lower} as $R\to\infty$, we obtain that
\begin{equation*}
\ln H_0(x)-\ln R_2\leq u(x)\leq \ln H_0(x)-\ln R_1
\end{equation*}
for any $x\in\Sigma\setminus B^{H_0}_{R_2}$. This readily implies $$\lim_{H_0(x)\to\infty}\frac{u(x)}{\ln H_0(x)}=1.$$
That is, \eqref{cap-eq:lim-gamma} holds with $\gamma=1$. From Proposition \ref{cap-pro:asym-equi}, we deduce that $u$ is a weak solution to problem \eqref{eq:log-cone} satisfying \eqref{cap-eq:thm-asym}.

Finally, let us show the uniqueness of $u$. Suppose that there is another function $v$ which solves problem \eqref{eq:log-cone} and satisfies $$\lim_{H_0(x)\to\infty}\left(v(x)-\ln H_0(x)\right)=\beta'$$ for some $\beta'\in\mathbb{R}$. Let $\beta_i$ $(i=1,2)$ be such that $\beta_1>\beta'$ and $\beta_2<\beta$. Consider the function $$\tilde{u}_R(x)=\frac{\ln R+\beta_1}{\ln R+\beta_2}u(x).$$ For sufficiently large $R$, $\tilde{u}_R\geq v$ on $\Gamma_0$ and $\Gamma_R$. Then by Lemma \ref{pre-lem:WCP}, $$\tilde{u}_R\geq v \quad\text{in }\Sigma\cap (B^{H_0}_R\setminus\overline{\Omega}).$$
Passing to the limit as $R\to\infty$, one has
$$u\geq v\quad\text{in } 
\Sigma\cap\overline{\Omega}^c.$$
Similarly, it also holds that $u\leq v$ in $\Sigma\cap\overline{\Omega}^c$ by taking $\beta_1<\beta'$ and $\beta_2>\beta$ and then applying Lemma \ref{pre-lem:WCP} again. We thus conclude $u=v$. This completes the proof.
\end{proof}

\section{Pohozaev identity}\label{sec:Pohozaev}

In this section, we derive the following Pohozaev-type identity for weak solutions of anisotropic $p$-Laplace equations with $p>1$. 
 
\begin{theorem}\label{P-thm:Pohozaev}
Let $\Omega\subset\mathbb{R}^N$ be a bounded open set with Lipschitz boundary and let $H$ be as in Theorem \ref{thm:over-log-cone}. Assume that $p>1$ and $u\in C^1(\overline{\Omega})$ satisfies
\begin{equation}\label{P-eq:distrib-Finsler}
\mathrm{div}\left(H^{p-1}(\nabla u)\nabla H(\nabla u)\right)=0\quad\text{in }\Omega.
\end{equation}
Then \begin{align*}
&\frac{p-N}{p}\int_{\Omega}H^p(\nabla u)\,dx\\
=&\int_{\partial\Omega}\left(H^{p-1}(\nabla u)\left<x,\nabla u\right>\left<\nabla H(\nabla u),\nu\right>-\frac{1}{p}H^p(\nabla u)\left<x,\nu\right>\right)d\mathcal{H}^{N-1}.
\end{align*}
\end{theorem}

Recall that $\nu$ denotes the unit outer normal to $\partial\Omega$. In the case that $\partial\Omega\in C^2$ and $H$ is the Euclidean norm, Theorem \ref{P-thm:Pohozaev} was established in \cite{Damascelli2009,Takac2012} for general $p$-Laplace equations of the form $-\Delta_p u=f$, under suitable regularity assumptions on $f$. In the present paper, for our applications on homogeneous equations, we generalize the argument as in \cite{Takac2012} to the anisotropic counterparts of $p$-harmonic functions.

We start with the following local version of the identity. 
\begin{lemma}\label{P-lem:loc-Pohozaev}
Let $1<p<\infty$ and $\Omega$ be an open subset of $\mathbb{R}^N$. Assume that  $u\in C^1(\Omega)$ satisfies \eqref{P-eq:distrib-Finsler}. Then
\begin{equation}\label{P-eq:loc-Pohozaev}
\mathrm{div}\left(\left<x,\nabla u\right>H^{p-1}(\nabla u)\nabla H(\nabla u)-\frac{1}{p}xH^p(
\nabla u)\right)=\frac{p-N}{p}H^p(\nabla u)
\end{equation}
holds in the sense of distributions in $\Omega$.
\end{lemma}
\begin{proof}
A direct computation shows that \eqref{P-eq:loc-Pohozaev} holds in the sense of distributions in the domain $\Omega\setminus Z$ where $$Z=\{x\in\Omega:\nabla u(x)=0\},$$ since $u\in W_{loc}^{2,2}(\Omega\setminus Z)$ by classical regularity theory for elliptic equations. Hence, we have to show that 
\begin{align}
&\int_{\Omega}\left(H^{p-1}(\nabla u)\left<x,\nabla u\right>\left<\nabla H(\nabla u),\nabla\varphi\right>-\frac{1}{p}H^p(
\nabla u)\left<x,\nabla\varphi\right>\right)dx\notag\\
&=\frac{N-p}{p}\int_{\Omega}H^p(\nabla u)\varphi\,dx\label{P-eq:distrib-proof}
\end{align}
holds for every $\varphi\in C_0^\infty(\Omega)$. 

From \cite[Theorem 4.2]{Avelin2018} (see also \cite{ACF}) we have $$H^{p-1}(\nabla u)\nabla H(\nabla u)\in W_{loc}^{1,2}(\Omega).$$For $\delta>0$ and $x\in\Omega$, let $$\psi_{\delta}(x)=\min\{\delta^{1-p}H^{p-1}(\nabla u(x)),1\}.$$ 
By \eqref{pre-eq:pro-H-con-1}, we have $H^{p-1}(\nabla u)=H_0(H^{p-1}(\nabla u)\nabla H(\nabla u))$. Since $H_0$ is Lipschitz continuous on $\mathbb{R}^N$, the chain rule entails $$H^{p-1}(\nabla u)\in W_{loc}^{1,2}(\Omega).$$ Thus, $\psi_\delta(x)\in W_{loc}^{1,2}(\Omega)$ and 
\begin{equation*}
\psi_{\delta}(x)=
\left\{\begin{array}{ll}
1 &  \text{if }x\in U_\delta:=\{x\in\Omega:H(\nabla u(x))\geq\delta\},\\
0 & \text{if }x\in Z.
\end{array}
\right.
\end{equation*}
Moreover, for every $x\in\Omega$, $\psi_\delta(x)\to\psi_0(x)$ as $\delta\to0$, where 
\begin{equation*}
\psi_0(x):=\left\{\begin{array}{ll}
1 &  \text{if }x\in \Omega\setminus Z,\\
0 & \text{if }x\in Z.
\end{array}
\right.
\end{equation*}
Given $\varphi\in C_0^\infty(\Omega)$, we thus can decompose it as the sum 
\begin{equation}\label{P-eq:phi-decompose}
\varphi=\psi_\delta\varphi+(1-\psi_\delta)\varphi
\end{equation}
 with $\psi_\delta\varphi \in W_{0}^{1,2}(\Omega\setminus Z)$ and $(1-\psi_\delta)\varphi\in W_0^{1,2}(\Omega\setminus U_\delta)$. Since $C_0^\infty(\Omega\setminus Z)$ is dense in $W_0^{1,2}(\Omega\setminus Z)$, \eqref{P-eq:distrib-proof} holds for $\psi_\delta\varphi$ in place of $\varphi$. Putting \eqref{P-eq:phi-decompose} into \eqref{P-eq:distrib-proof}, the left-hand side becomes
\begin{equation}\label{P-eq:left}
\frac{N-p}{p}\int_{\Omega}H^p(\nabla u)\psi_\delta\varphi\,dx+\int_{\Omega}\left<V_p,\nabla((1-\psi_\delta)\varphi)\right>dx
\end{equation}
where $$V_p=V_p(x):=H^{p-1}(\nabla u)\left<x,\nabla u\right>\nabla H(\nabla u)-\frac{1}{p}xH^p(
\nabla u).$$ For $\delta\in(0,1)$, we estimate
\begin{align}
&\int_{\Omega}\left<V_p,\nabla((1-\psi_\delta)\varphi)\right>dx\notag\\&\leq C(p) \int_{\Omega\setminus U_\delta}H^p(\nabla u)|x|\left(|\nabla\varphi|+|\varphi||\nabla\psi_\delta|\right)dx\notag\\
&\leq C(p)\delta\int_{\Omega\setminus U_\delta}\delta^{-p}H^p(\nabla u)|x|\left(\delta^{p-1}|\nabla\varphi|+\delta^{p-1}|\varphi||\nabla\psi_\delta|\right)dx\notag\\
&\leq C(p)\delta\int_{\Omega\setminus U_\delta}|x|\left(|\nabla\varphi|+\delta^{p-1}|\varphi||\nabla\psi_\delta|\right)dx.\label{P-eq:estimate-V}
\end{align}
Notice that the integrand in the last integral in \eqref{P-eq:estimate-V} is independent of $\delta$. Thus, by the dominated convergence theorem, we find 
\begin{equation}\label{P-eq:estimate-V-1}
\int_{\Omega}\left<V_p,\nabla((1-\psi_\delta)\varphi)\right>dx\to 0,\quad\text{as }\delta\to 0.
\end{equation}
Combining \eqref{P-eq:left} and \eqref{P-eq:estimate-V-1}, one deduces that 
\begin{align*}
\int_{\Omega}\left<V_p,\nabla \varphi\right>dx=\lim_{\delta\to 0}\int_{\Omega}\left<V_p,\nabla \varphi\right>dx&=\frac{N-p}{p}\lim_{\delta\to 0}\int_{\Omega}H^p(\nabla u)\psi_\delta\varphi\,dx\\
&=\frac{N-p}{p}\int_{\Omega}H^p(\nabla u)\psi_0\varphi\,dx.\\
&=\frac{N-p}{p}\int_{\Omega}H^p(\nabla u)\varphi\,dx.
\end{align*}
This shows \eqref{P-eq:distrib-proof}, thus completing the proof.
\end{proof}

To prove Theorem \ref{P-thm:Pohozaev}, we also need the following generalized version of the divergence theorem, and it will be used in proving Theorem \ref{thm:over-log-cone} as well in the next section. It is probably well-known, but we provide a proof for completeness.
 
\begin{lemma}\label{P-lem:divergence}
Let $\Omega$ be a bounded open subset of $\mathbb{R}^N$ with Lipschitz boundary and let $f\in L^1(\Omega)$. Assume that $\mathbf{a}\in C^0(\overline{\Omega};\mathbb{R}^N)$ satisfies $\mathrm{div}\,\mathbf{a}=f$ in the sense of distributions in $\Omega$. Then we have
\begin{equation*}
\int_{\partial\Omega}\left<\mathbf{a},\nu\right>d\mathcal{H}^{N-1}=\int_{\Omega}f(x)\,dx.
\end{equation*}
\end{lemma}

\begin{proof}
As in \cite[Lemma 2]{Squassina2003}, we let $k\geq1$ and $\phi_k:\mathbb{R}\to[0,1]$ be given by
\begin{equation*}
\phi_k(s)=\left\{
\begin{array}{ll}
0 &\text{if }s\leq\frac{1}{k},\\
ks-1 &\text{if }\frac{1}{k}<s<\frac{2}{k},\\
1 &\text{if }s\geq\frac{2}{k}.
\end{array}
\right.
\end{equation*}
Let $\psi_k\in C_0^{0,1}(\Omega)$ be given by $$\psi_k(x)=\phi_k(\mathrm{dist}(x,\mathbb{R}^N\setminus\Omega)).$$
From \cite[Sect. 7]{CDLP1988}, we have that $-\nabla\psi_k\to\nu\mathcal{H}^{N-1}\llcorner\partial\Omega$ weakly* in the sense of measures on $\overline{\Omega}$, namely,
\begin{equation}\label{P-eq:weak-star-normal}
\lim_{k\to\infty}\int_\Omega\left<v,\nabla\psi_k\right>\,dx=-\int_{\partial\Omega}\left<v,\nu\right>\,d\mathcal{H}^{N-1},\quad\forall v\in C^0(\overline{\Omega};\mathbb{R}^N).
\end{equation}
Also, $\lim_{k\to\infty}\psi_k=1$ for every $x\in\Omega$.

Since $C_0^{0,1}(\Omega)\hookrightarrow W_0^{1,q}(\Omega)\hookrightarrow L^\infty(\Omega)$ for $q>N$. By the density of $C_0^\infty(\Omega)$ in $W_0^{1,q}(\Omega)$, we can test $\mathrm{div}\,\mathbf{a}=f$ with $\psi_k$ to get
\begin{equation}\label{P-eq:test-divergence}
-\int_\Omega\left<\mathbf{a},\nabla\psi_k\right>\,dx=\int_\Omega f\psi_k\,dx.
\end{equation}
Hence, in view of \eqref{P-eq:weak-star-normal}, the assertion follows by passing to the limit in \eqref{P-eq:test-divergence} as $k\to\infty$ and using the dominated convergence theorem. 
\end{proof}

\begin{proof}[Proof of Theorem \ref{P-thm:Pohozaev}]
Applying Lemma \ref{P-lem:loc-Pohozaev} together with Lemma \ref{P-lem:divergence}, one can immediately get the desired integral identity.
\end{proof}

\section{Proof of Theorem \ref{thm:over-log-cone}}\label{sec:proof}
In this section we prove Theorem \ref{thm:over-log-cone}. At the end of the section, we also present the proof of  Proposition \ref{pro:observe} although it follows easily via Lemmas \ref{pre-lem:pro-H} and \ref{pre-lem:pro-H-con}.

We start by showing that the solvability of overdetermined problem \eqref{eq:log-cone}--\eqref{eq:asym-log-cone} with \eqref{eq:over-Neumann} implies a priori relation between the value of $C$ given in \eqref{eq:over-Neumann} and the anisotropic perimeter of the set $\Omega$ relative to $\Sigma$. Throughout the section, we denote $\Gamma_R:=\Sigma\cap\partial B_R^{H_0}$ for $R>1$, where $B_R^{H_0}:=B_R^{H_0}(0)$.
\begin{lemma}\label{proof-lem:C}
Let $\Sigma$, $\Omega$, $H$ and $H_0$  be as in Theorem \ref{thm:over-log-cone}. Assume that problem \eqref{eq:log-cone}--\eqref{eq:asym-log-cone} with \eqref{eq:over-Neumann} admits a weak solution $u$ satisfying $\nabla u\in L_{loc}^{\infty}(\overline{\Sigma}\cap\Omega^c)$. Then $$\left(\frac{C}{\gamma}\right)^{N-1}=\frac{P_H(B^{H_0}_1;\Sigma)}{P_H(\Omega;\Sigma)},$$
where $\gamma$ is as in Proposition \ref{cap-pro:asym-equi}.
\end{lemma}

\begin{proof}
We can assume $\gamma=1$ (otherwise consider $u/\gamma$). Let $R>1$ be such that $B^{H_0}_R\supset\overline{\Omega}$. The idea is to obtain an integral identity on $\Gamma_0\cup\Gamma_R$ by applying Lemma \ref{P-lem:divergence}, and then let $R\to\infty$ in the identity by exploiting the asymptotic behavior at infinity of $u$ given in Proposition \ref{cap-pro:asym-equi}. Since $u$ is not $C^1$ up to the whole boundary, we will argue by approximation. 

Without loss of generality, we assume that the $N$-th direction vector $e_n$ belongs to $\Sigma$ and $$\Sigma=\{x_N>g(x_1,\cdots,x_{N-1})\}$$
for some convex function $g:\mathbb{R}^{N-1}\to\mathbb{R}$. From \cite[Proof of Theorem 1.3]{CRS2016}, there exists a sequence of smooth convex functions $g_k:\mathbb{R}^{N-1}\to\mathbb{R}$ with $k\in\mathbb{N}\setminus\{0\}$ such that the sets given by $$F_k:=\{x_N>g_k(x_1,\cdots,x_{N-1})\}$$
belong to $\Sigma$ and satisfy that $F_k\cap(B_R^{H_0}\setminus\overline{\Omega})$ are Lipschitz ($\partial F_k$ and $\partial(B_R^{H_0}\setminus\overline{\Omega})$ intersect transversally) and they approximate $\Sigma\cap(B_R^{H_0}\setminus\overline{\Omega})$ in the $L^1$ sense. In particular, the functions $g_k$ fulfill:
\begin{itemize}
\item[(i)] $g_{k+1}<g_{k}$ in $\overline{B}$ where $B\subset\mathbb{R}^{N-1}$ is a large ball containing the projection of $\overline{B_R^{H_0}}\setminus\Omega$;
\item[(ii)] $g_k\to g$ uniformly in $\overline{B}$;
\item[(iii)] $|\nabla g_k|$ is uniformly bounded and $\nabla g_k\to\nabla g$ a.e. in $\overline{B}$.  
\end{itemize}

Notice that $u\in C^1\left(\overline{F_k\cap(B_R^{H_0}\setminus\overline{\Omega})}\right)$. From Lemma \ref{P-lem:divergence} we have
\begin{equation}\label{proof-eq:local-diver}
\int_{\partial F_k\cap(B_R^{H_0}\setminus\overline{\Omega})}\left<a(\nabla u),\nu\right>d\mathcal{H}^{N-1}+\int_{F_k\cap\partial(B_R^{H_0}\setminus\overline{\Omega})}\left<a(\nabla u),\nu\right>d\mathcal{H}^{N-1}=0.
\end{equation}
 By Proposition \ref{pre-pro:second-order}, we see $\left<a(\nabla u),\nu\right>=0$ a.e. on $\partial\Sigma\cap(B_R^{H_0}\setminus\overline{\Omega})$. Since $\nabla u\in C^0\cap L^\infty(\Sigma\cap(B_R^{H_0}\setminus\overline{\Omega}))$, by using dominated convergence and properties (ii)--(iii) we deduce that
\begin{equation}\label{proof-eq:appr-k-cone}
\int_{\partial F_k\cap(B_R^{H_0}\setminus\overline{\Omega})}\left<a(\nabla u),\nu\right>d\mathcal{H}^{N-1}\to\int_{\partial \Sigma\cap(B_R^{H_0}\setminus\overline{\Omega})}\left<a(\nabla u),\nu\right>d\mathcal{H}^{N-1}=0.
\end{equation}
Moreover, by properties (i)--(ii) we obtain
\begin{equation}\label{proof-eq:appr-k-O}
\int_{F_k\cap\partial(B_R^{H_0}\setminus\overline{\Omega})}\left<a(\nabla u),\nu\right>d\mathcal{H}^{N-1}\to\int_{\Sigma\cap\partial(B_R^{H_0}\setminus\overline{\Omega})}\left<a(\nabla u),\nu\right>d\mathcal{H}^{N-1}.
\end{equation}
Hence, by combining \eqref{proof-eq:local-diver} with \eqref{proof-eq:appr-k-cone} and \eqref{proof-eq:appr-k-O}, we  arrive at
$$\int_{\Sigma\cap\partial(B_R^{H_0}\setminus\overline{\Omega})}\left<a(\nabla u),\nu\right>d\mathcal{H}^{N-1}=0.$$

Now, from \eqref{cap-eq:asym-D-equi}, we infer that
\begin{gather*}
\nabla u=\nabla(\ln H_0(x))+o(H_0^{-1}(x))\\
H(\nabla u)=H_0^{-1}(x)+o(H_0^{-1}(x))
\end{gather*}
uniformly for $x\in\Gamma_R$, as $R\to\infty$. So that
\begin{align*}
&\int_{\Gamma_0}\left<a(\nabla u),\nu\right>d\mathcal{H}^{N-1}\\
&=\int_{\Gamma_R}\left<a(\nabla u),\nu\right>d\mathcal{H}^{N-1}\\
&=\int_{\Gamma_R}\left<a(\nabla u),\frac{\nabla H_0(x)}{|\nabla H_0(x)|}\right>d\mathcal{H}^{N-1}\\
&=R\int_{\Gamma_R}\left<a(\nabla u),\frac{\nabla u+o(H_0^{-1}(x))}{|\nabla H_0(x)|}\right>d\mathcal{H}^{N-1}\\
&=R\int_{\Gamma_R}H^{N-1}(\nabla u)\frac{H(\nabla u)+o(R^{-1})}{|\nabla H_0(x)|}d\mathcal{H}^{N-1}\\
&=\int_{\Gamma_R}\left(R^{-1}+o(R^{-1})\right)^{N-1}(1+o(1))H(\nu)\,d\mathcal{H}^{N-1}.
\end{align*} 
Letting $R
\to\infty$, we obtain that the right-hand side of the above equality becomes $P_H(B^{H_0}_1;\Sigma)$. On the other hand, since $\nu=\frac{\nabla u}{|\nabla u|}$ on $\Gamma_0$, the left-hand side is 
\begin{equation*}
\int_{\Gamma_0}H^{N-1}(\nabla u)\left<\nabla H(\nabla u),\frac{\nabla u}{|\nabla u|}\right>d\mathcal{H}^{N-1}=C^{N-1}P_H(\Omega;\Sigma).
\end{equation*}
Consequently, we get $C^{N-1}P_H(\Omega;\Sigma)=P_H(B^{H_0}_1;\Sigma)$ if $\gamma=1$. Then the conclusion for general $\gamma>0$ follows easily. This completes the proof.
\end{proof}

Using Lemma \ref{proof-lem:C} and Theorem \ref{P-thm:Pohozaev}, together with Theorem \ref{pre-thm:Wulff-cone}, we are now in position to prove Theorem \ref{thm:over-log-cone}. 
\begin{proof}[Proof of Theorem \ref{thm:over-log-cone}]
Let $$V_N(x)=H^{N-1}(\nabla u)\left<x,\nabla u\right>\nabla H(\nabla u)-\frac{1}{N}xH^N(
\nabla u).$$
Notice that $\left<V_N,\nu\right>=0$ a.e. on $\Gamma_1$. For any $R>1$ such that $B^{H_0}_R\supset\overline{\Omega}$, we aim at applying Theorem \ref{P-thm:Pohozaev} to equation \eqref{eq:log-cone} in $\Sigma\cap(B_R^{H_0}\setminus\overline\Omega)$. Due to the lack of regularity of $u$, we approximate $\Sigma\cap(B_R^{H_0}\setminus\overline\Omega)$ by a sequence of Lipschitz domains $F_k\cap(B_R^{H_0}\setminus\overline{\Omega})$ as done in the proof of Lemma \ref{proof-lem:C}. Since $u\in C^1\left(\overline{F_k\cap(B_R^{H_0}\setminus\overline{\Omega})}\right)$, by applying Theorem \ref{P-thm:Pohozaev} we get 
\begin{equation}\label{proof-eq:local-Poho}
\int_{\partial F_k\cap(B_R^{H_0}\setminus\overline{\Omega})}\left<V_N,\nu\right>d\mathcal{H}^{N-1}+\int_{F_k\cap\partial(B_R^{H_0}\setminus\overline{\Omega})}\left<V_N,\nu\right>d\mathcal{H}^{N-1}=0.
\end{equation}
Furthermore, by arguing as in the proof of Lemma \ref{proof-lem:C}, we can take the limit of \eqref{proof-eq:local-Poho} as $k\to\infty$ to deduce that
\begin{equation}\label{proof-eq:P-V-k}
\int_{\Gamma_0}\left<V_N,\nu\right>d\mathcal{H}^{N-1}-\int_{\Gamma_R}\left<V_N,\nu\right>d\mathcal{H}^{N-1}=0.
\end{equation}

On the one hand, using the fact that $\nu=\frac{\nabla u}{|\nabla u|}$ on $\Gamma_0$ and that $\int_{\Gamma_0}\left<x,\nu\right>=N\mathcal{H}^N(\Sigma\cap\Omega)$ yields 
\begin{align*}
&\int_{\Gamma_0}\left<V_N,\nu\right>d\mathcal{H}^{N-1}\\
&=\int_{\Gamma_0}\left(H^{N}(\nabla u)\left<x,\nu\right>-\frac{H^N(\nabla u)}{N}\left<x,\nu\right>\right)d\mathcal{H}^{N-1}\\
&=(N-1)C^N\mathcal{H}^N(\Sigma\cap\Omega).
\end{align*}
On the other hand, via Proposition \ref{cap-pro:asym-equi} we see there exists $\gamma>0$ such that
\begin{equation*}
H(\nabla u)=\gamma H_0^{-1}(x)+o(H_0^{-1}(x))\quad\text{and}\quad
\left<x,\nabla u\right>=\gamma+o(1),
\end{equation*}
uniformly for $x\in\Gamma_R$, as $R\to\infty$. Similar to the proof of Lemma \ref{proof-lem:C}, a direct computation entails
$$
\int_{\Gamma_R}\left<V_N,\nu\right>d\mathcal{H}^{N-1}\to\frac{\gamma^N(N-1)}{N}P_H(B^{H_0}_1;\Sigma),
$$
as $R\to\infty$.

Consequently, from \eqref{proof-eq:P-V-k} we arrive at 
$$C^N\mathcal{H}^N(\Sigma\cap\Omega)=\frac{\gamma^N}{N}P_H(B^{H_0}_1;\Sigma).$$
By recalling that $P_H(B^{H_0}_1;\Sigma)=N\mathcal{H}^N(\Sigma\cap B_1^{H_0})$ (see \cite[Formula (1.14)]{CRS2016}) and using
Lemma \ref{proof-lem:C}, we conclude
$$\frac{P_{H}(\Omega;\Sigma)}{\mathcal{H}^N(\Sigma\cap\Omega)^{\frac{N-1}{N}}}= \frac{P_H(B_1^{H_0};\Sigma)}{\mathcal{H}^N(\Sigma\cap B_1^{H_0})^{\frac{N-1}{N}}}.$$ That is, $\Sigma\cap\Omega$ satisfies the equality case of Theorem \ref{pre-thm:Wulff-cone}. This forces that $\Sigma\cap\Omega=\Sigma\cap B_R^{H_0}(x_0)$ for some $R>0$ and $x_0\in\overline{\Sigma}$ as described in the theorem. Then the representation of $u$ and its uniqueness are concluded from Proposition \ref{pro:observe}. This finishes the proof.
\end{proof}

\begin{proof}[Proof of Proposition \ref{pro:observe}]
Let $u$ be as in \eqref{intro-eq:formula-u}. First, $u=0$ on $\Gamma_0$ holds trivially since $H_0(x-x_0)=R$ on $\Gamma_0$. From Lemma \ref{pre-lem:pro-H-con}, we see that $u\in C^2(\mathbb{R}^N\setminus\{x_0\})$ and $$\nabla u=CR\frac{\nabla H_0(x-x_0)}{H_0(x-x_0)}.$$
By \eqref{pre-eq:pro-H-con-1}, $$H(\nabla u)=\frac{CR}{H_0(x-x_0)}=C\quad\text{on }\Gamma_0,$$
that is, the condition \eqref{eq:over-Neumann} holds. Moreover, by Lemma \ref{pre-lem:pro-H} and \eqref{pre-eq:pro-H-con-1}--\eqref{pre-eq:pro-H-con-2}, we get 
\begin{align*}
a(\nabla u)&=(CR)^{N-1}H_0^{1-N}(x-x_0)\nabla H(\nabla H_0(x-x_0))\\
&=(CR)^{N-1}H_0^{1-N}(x-x_0)\frac{x-x_0}{H_0(x-x_0)}\\
&=(CR)^{N-1}\frac{x-x_0}{H_0^N(x-x_0)}.
\end{align*}
Clearly, $\left<a(\nabla u),\nu\right>=\left<x-x_0,\nu\right>=0$ a.e. on $\Gamma_1$. Next, let us verify that $u$ satisfies the equation \eqref{eq:log-cone}. It follows from Lemma \ref{pre-lem:pro-H} that
\begin{align*}
&\mathrm{div}\,(\frac{x-x_0}{H_0^N(x-x_0)})\\
=&H_0^{-N}(x-x_0)\mathrm{div}(x-x_0)+\left<x-x_0,(-N)H_0^{-N-1}\nabla H_0(x-x_0)\right>\\
=&NH_0^{-N}(x-x_0)-NH_0^{-N}(x-x_0)=0.
\end{align*}
This implies $\Delta_N^H u=\mathrm{div}\,(a(\nabla u))=0$ for $x\neq x_0$. 

Finally, the uniqueness of $u$ follows from Theorem \ref{cap-thm:F-log}.
\end{proof}

\section*{Acknowledgements}
The first author has been partially supported by the ``Gruppo Nazionale per l'Analisi Matematica, la Probabilit\`a e le loro Applicazioni'' (GNAMPA) of the ``Istituto Nazionale di Alta Matematica'' (INdAM, Italy). The second author is supported by China Scholarship Council. This work has been done while the second author was visiting the Department of Mathematics ``Federigo Enriques'' of Universit\`a degli Studi di Milano, which is acknowledged for the hospitality. 

The authors thank Giorgio Poggesi for bringing up to their attention the reference \cite{Xia-Yin-2021}.

\end{document}